\documentclass[11pt]{amsart}
\usepackage{amsmath,amssymb,hyperref}
\usepackage{amsfonts}
\usepackage{enumerate}
\usepackage{amsmath, amsthm}
\usepackage{amscd}
\input xy
\xyoption{all}
\hypersetup{pdfpagemode=FullScreen, colorlinks=true}

\theoremstyle{plain}
\newtheorem{thm}{Theorem}[section]

\newtheorem{prop}[thm]{Proposition}
\newtheorem{lemma}[thm]{Lemma}

\newtheorem{cor}[thm]{Corollary}

\theoremstyle{definition}
\newtheorem{defi}[thm]{Definition}

\newtheorem{rem}[thm]{Remark}

\newtheorem{conj}[thm]{Conjecture}

\newcommand{\bc}{{\mathbb C}}
\newcommand{\bh}{{\mathbb H}}
\newcommand{\br}{{\mathbb R}}

\newcommand{\bz}{{\mathbb Z}}
\newcommand{\ra}{\rightarrow}

\numberwithin{equation}{section}
\def\co{\colon\thinspace}
\let \cal \mathcal

\begin{document}

\title[Volume invariant and maximal representation]
{Volume invariant and maximal representations of discrete subgroups of Lie groups}

\author{Sungwoon Kim}
\address{School of Mathematics,
KIAS, Hoegiro 85, Dongdaemun-gu, Seoul, 130-722, Republic of Korea}
\email{sungwoon@kias.re.kr}

\author{Inkang Kim}
\address{School of Mathematics
KIAS, Hoegiro 85, Dongdaemun-gu, Seoul, 130-722, Republic of Korea}
\email{inkang@kias.re.kr}

\footnotetext[3]{The second
 author gratefully acknowledges the partial support of KRF grant
(0409-20060066).}

\begin{abstract}
Let $\Gamma$ be a lattice in a connected semisimple Lie group $G$ with trivial center and no compact factors.
We introduce a volume invariant for representations of $\Gamma$ into
$G$, which generalizes the volume invariant for representations of uniform lattices introduced by Goldman.
Then, we show that  the maximality of this volume invariant exactly characterizes discrete, faithful representations of $\Gamma$ into $G$.
\end{abstract}

\maketitle

\section{Introduction}

A volume invariant is defined to characterize discrete, faithful representations of a discrete group $\Gamma$
into a connected semisimple Lie group $G$. For a uniform lattice $\Gamma$, Goldman \cite{Go92} introduced a volume invariant $\upsilon(\rho)$ of a representation $\rho \co \Gamma \rightarrow G$ as follows:
Let $X$ be the associated symmetric space of dimension $n$ and $M=\Gamma\backslash X$.
To every representation $\rho \co \Gamma \rightarrow G$, a bundle $E_\rho$ over $M$ with fibre $X$ and structure group $G$ is associated.
One can obtain a closed $n$--form $\omega_\rho$ on $E_\rho$ by
spreading the $G$--invariant volume form $\omega$ on $X$ over the fibres of $E_\rho$. Then, the volume invariant $\upsilon(\rho)$ of $\rho$ is defined by
$$ \upsilon (\rho) = \int_M f^*\omega_\rho,$$
where $f$ is a section of $E_\rho$.

The definition of the volume invariant $\upsilon(\rho)$ is independent of the choice of a section since $X$ is contractible. It can be easily seen that the volume invariant $\upsilon(\rho)$ satisfies an inequality
\begin{eqnarray}\label{eqn:1.1} |\upsilon(\rho)| \leq \mathrm{Vol}(M),\end{eqnarray} which recovers the Milnor-Wood inequality
for $G=\mathrm{PSL}_2(\mathbb{R})$. Note that the volume invariant $\upsilon(\rho)$ is available only for representations of uniform lattices.
Goldman \cite{Go92} conjectured the following and gave a positive answer for all connected semisimple Lie groups except for $\mathrm{SU}(n,1), \mathrm{Sp}(n,1), \mathrm{F}_4^{-20}$.

\begin{conj}\label{con:1.1}
Equality holds in (\ref{eqn:1.1}) if and only if $\rho$ is a discrete, faithful representation
of $\Gamma$ into $G$.
\end{conj}

Numerical invariants such as the volume invariant have been used to study a representation variety $\mathrm{Hom}(\Gamma,G)$ consisting of homomorphisms $\rho \co \Gamma \rightarrow G$.
For example, Goldman \cite{Go88} characterized $(4g-3)$--connected components of the representation variety $\mathrm{Hom}(\pi_1(S),\mathrm{PSL}_2 (\mathbb{R}))$ for a closed surface $S$ of genus $g$ via the Toledo invariant.
Moreover, he verified that the connected component of $\mathrm{Hom}(\pi_1(S),\mathrm{PSL}_2 (\mathbb{R}))$ with maximal Toledo invariant is exactly the embedding of the Teichm\"{u}ller space of $S$ into $\mathrm{Hom}(\pi_1(S),\mathrm{PSL}_2 \mathbb{R})$ \cite{Go80}.
Burger, Iozzi and Wienhard \cite{BIW10} generalize the theories of a closed surface representation variety in $\mathrm{PSL}_2 (\mathbb{R})$ to other Lie groups such as split simple Lie groups and Lie groups of Hermitian type.

In comparison with uniform lattices, numerical invariants for representations of nonuniform lattices have been rarely defined. The main reason for this is that the fundamental class of open manifolds vanishes in the top dimensional singular homology. Recently, Burger, Iozzi and Wienhard \cite{BIW11} define
the Toledo invariant for representations of a compact surface with boundary by using its relative
fundamental class. Then, they show that this Toledo invariant exactly detects hyperbolic structures on the surface.

The aim of this paper is to introduce a new invariant for representations of arbitrary lattices $\Gamma$ in $G$ which detects discrete, faithful representations in the representation variety $\mathrm{Hom}(\Gamma,G)$. One advantage of the new invariant is that it provides a tool for studying the representation varieties of nonuniform lattices in semisimple Lie groups. In addition, we explore the relation between the new invariant and $\upsilon(\rho)$. Then, we give a proof of Conjecture \ref{con:1.1}.

Let $\Gamma$ be a lattice in $G$. Every representation $\rho \co \Gamma \rightarrow G$ induces canonical pullback maps $\rho^*_b \co H^\bullet_{c,b}(G,\mathbb{R})\rightarrow H^\bullet_b(\Gamma,\mathbb{R})$ in continuous bounded cohomology.
Let $c \co H^\bullet_{c,b}(G,\mathbb{R})\rightarrow H^\bullet_c(G,\mathbb{R})$ be the comparison map induced
from the inclusion of the continuous bounded cochain complex of $G$ into the continuous cochain complex of $G$.
The Van Est isomorphism gives an isomorphism $H^n_c(G,\mathbb{R})\cong \mathbb{R}\cdot \omega$,
where $\omega$ is the $G$--invariant volume form on the associated symmetric space $X$.
Then, we define a new invariant $\mathrm{Vol}(\rho)$ by
$$\mathrm{Vol}(\rho) = \inf \{ |\langle \rho^*_b(\omega_b),\alpha \rangle| \text{ }|\text{ }c(\omega_b)=\omega \text{ and } \alpha \in [M]^{\ell^1}_\mathrm{Lip} \},
$$
where $[M]^{\ell^1}_\mathrm{Lip}$ is the set of all $\ell^1$--homology classes in $H^{\ell^1}_n(M,\mathbb{R})$ that are represented by at least one locally finite fundamental cycle with finite Lipschitz constant.
Note that $\rho^*_b (\omega_b)$ is regarded as a bounded
cohomology class in $H^n_b(M,\mathbb{R})$ by the canonical isomorphism between $H^n_b(\Gamma,\mathbb{R})$ and $H^n_b(M,\mathbb{R})$. Thus, $\rho^*_b(\omega_b)$ can be evaluated on $\ell^1$--homology classes in $H^{\ell^1}_n(M,\mathbb{R})$ and hence, the definition of $\mathrm{Vol}(\rho)$ makes sense.
For more details on the definition and properties of the volume invariant $\mathrm{Vol}(\rho)$, see Section \ref{sec:3}.

An essential ingredient in defining the volume invariant $\mathrm{Vol}(\rho)$ is the geometric simplicial volume of $M$, introduced by Gromov \cite{Gr82}.
Indeed, Gromov defined two kinds of simplicial volumes for open Riemannian manifolds. One is defined as the $\ell^1$--seminorm of the locally finite fundamental class of $M$. This is a topological invariant. The other is defined by the infimum over all $\ell^1$--norms of locally finite fundamental cycles of $M$ with finite Lipschitz constant.
The latter is called the geometric simplicial volume of $M$ because the Riemannian structure on $M$ is involved in its definition. Note that this is not a topological invariant anymore.

One can notice that the volume invariant $\mathrm{Vol}(\rho)$ can be defined via
locally finite fundamental cycles of $M$ instead of locally finite fundamental cycles with finite Lipschitz constant.
However, it turns out that if the volume invariant $\mathrm{Vol}(\rho)$ is defined via locally finite fundamental cycles, then this invariant does not always detect discrete, faithful representations. For further discussion of this, see Section \ref{sec:3.1}.

\begin{thm}\label{thm:1.2}
Let $\Gamma$ be an irreducible lattice in a connected semisimple Lie group $G$ with trivial center and
no compact factors. Let $\rho \co \Gamma \rightarrow G$ be a representation. Then, the volume invariant $\mathrm{Vol}(\rho)$ satisfies an inequality
$$ \mathrm{Vol}(\rho) \leq \mathrm{Vol}(M),$$
where $X$ is the associated symmetric space and $M=\Gamma\backslash X$.
Moreover,  equality holds if and only if
$\rho$ is a discrete, faithful representation.
\end{thm}

Theorem \ref{thm:1.2} implies that the volume invariant $\mathrm{Vol}(\rho)$ exactly characterizes discrete, faithful representations
in the representation variety $\mathrm{Hom}(\Gamma,G)$. In particular, when $\Gamma$ is a uniform lattice, we verify that
\begin{eqnarray}\label{eqn:1.2} \mathrm{Vol}(\rho)=|\upsilon(\rho)|. \end{eqnarray}
From the view of Equation (\ref{eqn:1.2}), the volume invariant $\mathrm{Vol}(\rho)$ can be regarded as an invariant for representations of arbitrary lattices extending the volume invariant $\upsilon(\rho)$ only for representations of uniform lattices. Note that Theorem \ref{thm:1.2} covers the remaining cases $\mathrm{SU}(n,1), \mathrm{Sp}(n,1), \mathrm{F}_4^{-20}$ that Goldman's proof in \cite{Go92} did not cover. In fact, one can easily notice that Conjecture \ref{con:1.1} is able to be proved by using the Besson-Courtois-Gallot technique in \cite{BCG99}.

In a similar way, we define a volume invariant $\mathrm{Vol}(\rho)$ for representations $\rho \co \Gamma \rightarrow \mathrm{SO}(m,1)$ of lattices $\Gamma$ in $\mathrm{SO}(n,1)$. A representation $\rho \co \Gamma \rightarrow \mathrm{SO}(m,1)$ is said to be a \emph{totally geodesic representation} if there is a totally geodesic $\mathbb{H}^n \subset \mathbb{H}^m$ so that the image of the representation lies in the subgroup $G \subset \mathrm{SO}(m,1)$ that preserves this $\mathbb{H}^n$ and that the $\rho$--equivariant map $F \co \mathbb{H}^n \rightarrow \mathbb{H}^m$ is a totally geodesic isometric embedding. Then, we show that this volume invariant characterizes totally geodesic representations.

\begin{thm}
Let $\Gamma$ be a  lattice in $\mathrm{SO}(n,1)$ and $M=\Gamma \backslash \mathbb{H}^n$. The volume invariant $\mathrm{Vol}(\rho)$ of a representation $\rho \co \Gamma \rightarrow \mathrm{SO}(m,1)$ for $m\geq n \geq 3$ satisfies an inequality $$\mathrm{Vol}(\rho) \leq \mathrm{Vol}(M).$$
Moreover, equality holds if and only if $\rho$ is a totally geodesic representation.
\end{thm}

Finally using bounded cohomology theory and volume invariant,
we can formulate the local rigidity phenomena of  complex hyperbolic uniform lattices. Specially we prove that
\begin{thm}Let $\Gamma\subset \mathrm{SU}(n,1)$ be a uniform lattice and $\rho:\Gamma\ra \mathrm{SU}(m,1)$, $m\geq n \geq 2$  a representation. Then it is
a maximal volume representation if and only if it is a totally geodesic representation. For the natural inclusion
$\Gamma\subset \mathrm{SU}
(n,1)\subset \mathrm{SU}(m,1)\subset  \mathrm{Sp}(m,1)$, it is locally rigid, in the sense that the nearby representations stabilize a copy of $\bh^n_\bc$ inside $\bh^m_\bh$.
\end{thm}
This paper is organized as follows: We review the simplicial volume, $\ell^1$--homology and continuous (bounded) cohomology in order to define the new invariant $\mathrm{Vol}(\rho)$ in Section \ref{sec:2}. We describe the basic properties of the volume invariant $\mathrm{Vol}(\rho)$ in Section \ref{sec:3}. Then, we devote ourselves to proving Theorem \ref{thm:1.2} for the case that $G$ is a semisimple Lie group of higher rank in Section \ref{sec:4}, $G$ is a simple Lie group of rank $1$ except for $\text{SO}(2,1)$ in Section \ref{sec:5} and $G$ is $\mathrm{SO}(2,1)$ in Section \ref{sec:6}.
We deal with a volume invariant for representations $\rho \co \Gamma \rightarrow \mathrm{SO}(m,1)$ of lattices $\Gamma$ in $\mathrm{SO}(n,1)$ in Section \ref{sec:7}. Lastly, we reformulate the rigidity phenomenon of uniform lattices of $\mathrm{SU}(n,1)$ in $\mathrm{SU}(m,1)$ or $\mathrm{Sp}(m,1)$ via the volume invariant in Section \ref{sec:8}.

\section{Preliminaries}\label{sec:2}

\subsection{Simplicial volume}

Let $M$ be an $n$--dimensional manifold. The simplicial $\ell^1$--norm $\| \cdot \|_1$ on the singular chain complex $C_\bullet(M,\mathbb{R})$ is defined by
the $\ell^1$--norm with respect to the basis given by all singular simplices.
The simplicial $\ell^1$--norm induces a $\ell^1$--seminorm on $H_\bullet(M,\mathbb{R})$ as follows:
$$\| \alpha \|_1 =\inf \| c \|_1$$
where $c$ runs over all singular cycles representing $\alpha \in H_\bullet(M,\mathbb{R})$.

For an oriented, connected, closed $n$--manifold $M$, the simplicial volume $\| M \|$ of $M$ is defined as the $\ell^1$--seminorm of the fundamental class $[M]$ in $H_n(M,\mathbb{R})$.
If $M$ is an oriented, connected, open $n$--manifold, then $M$ has a fundamental class $[M]$ in the locally finite homology $H^\mathrm{lf}_n(M,\mathbb{R})$.
The locally finite homology of $M$ is defined as the homology of the locally finite chain complex $C_\bullet^\mathrm{lf}(M,\mathbb{R})$.
More precisely, let $S_k(M)$ be the set of singular $k$--simplices of $M$ and $S^\mathrm{lf}_k(M)$ denote the set of all locally finite subsets of $S_k(M)$, that is, if $A \in S^\mathrm{lf}_k(M)$, any compact subset of $M$ intersects the image of only finitely many elements of $A$. Then, the locally finite chain complex $C_\bullet^\mathrm{lf}(M,\mathbb{R})$ is defined by
$$C_\bullet^\mathrm{lf}(M,\mathbb{R})= \left\{\sum_{\sigma \in A} a_\sigma \sigma \ \bigg| \ A \in S^\mathrm{lf}_\bullet(X) \text{ and }a_\sigma \in \mathbb{R} \right\}.$$

A $\ell^1$--seminorm on $H^\mathrm{lf}_\bullet(M,\mathbb{R})$ is
induced from the simplicial $\ell^1$--norm on the locally finite chain complex $C^\mathrm{lf}_\bullet(M,\mathbb{R})$ with respect to the basis given by all singular simplices. The simplicial volume $\| M \|$ of $M$ is defined as the $\ell^1$--seminorm of the locally finite fundamental class $[M]$ of $M$.

In addition, Gromov introduces the geometric simplicial volume of oriented, connected, open Riemannian manifolds. Fixing a metric on the standard $k$--simplex $\Delta^k$ by the Euclidean metric, the Lipschitz constant $\mathrm{Lip}(\sigma)$ of a singular simplex $\sigma \co \Delta^k \rightarrow M$ is defined. Subsequently, for a locally finite chain $c \in C^\mathrm{lf}_\bullet(M,\mathbb{R})$, define the Lipschitz constant $\mathrm{Lip}(c)$ of $c$ by the supremum over all Lipschitz constants of the simplices occurring in $c$.

The subcomplex $C^\mathrm{lf,Lip}_\bullet(M,\mathbb{R})$ of $C^\mathrm{lf}_\bullet(M,\mathbb{R})$ consisting of all chains with finite Lipschitz constant
induces the homology with Lipschitz locally finite support, denoted by $H^\mathrm{lf,Lip}_\bullet(M,\mathbb{R})$. Indeed, $H^\mathrm{lf,Lip}_\bullet(M,\mathbb{R})$ is isomorphic to $H^\mathrm{lf}_\bullet(M,\mathbb{R})$ \cite[Theorem 3.3]{LS09}. Hence, it has a distinguished generator $[M]_\mathrm{Lip}$ in $H^\mathrm{lf,Lip}_\bullet(M,\mathbb{R})$ corresponding to the locally finite fundamental class $[M]$ in $H^\mathrm{lf}_\bullet(M,\mathbb{R})$. The geometric simplicial volume of $M$ is defined as the $\ell^1$--seminorm of $[M]_\mathrm{Lip}$, denoted by $\| M \|_\mathrm{Lip}$.
Gromov \cite{Gr82} proves the proportionality principle for the geometric simplicial volume as follows.

\begin{thm}[Gromov]\label{thm:2.1}
Let $M$ be a closed Riemannian manifold and $N$ be a complete Riemannian manifold of
finite volume. If the universal covers of $M$ and $N$ are isometric, then
$$\frac{\|M\|_\mathrm{Lip}}{\mathrm{Vol}(M)}=\frac{\|N\|_\mathrm{Lip}}{\mathrm{Vol}(N)}.$$
\end{thm}

The simplicial volume of a smooth manifold gives a lower bound of its minimal volume. Hence,
the question was naturally raised as to which manifolds have nonzero simplicial volumes. Gromov \cite{Gr82} and Thurston \cite{Th78} first
show that the simplicial volume of complete Riemannian manifolds of finite volume with pinched negative sectional curvature is nonzero.
Moreover, it is shown that closed locally symmetric spaces of noncompact type have positive simplicial volumes \cite{LS06}.
In contrast, the simplicial volume of open, complete locally symmetric spaces of noncompact type with finite volume may vanish. For instance,
the simplicial volume of locally symmetric spaces of noncompact type with $\mathbb{Q}$--rank at least $3$ vanishes \cite{LS09}. On the other hand, it turns out that the simplicial volume of $\mathbb{Q}$--rank $1$ locally symmetric spaces covered by a product of $\mathbb{R}$--rank $1$ symmetric spaces is positive \cite{KK} and moreover, it is equal to their geometric simplicial volume \cite{BKK} for amenable boundary group cases. The $\mathbb{Q}$--rank $2$ cases remain open.

\subsection{$\ell^1$--homology}

Let $M$ be an oriented, connected $n$--manifold.
The $\ell^1$--chain complex of $M$ is the $\ell^1$--completion  $C_\bullet^{\ell^1}(M,\mathbb{R})$  of
the normed chain complex $C_\bullet(M,\mathbb{R})$ with respect to the simplicial $\ell^1$--norm $\| \cdot \|_1$. Then, the $\ell^1$--homology $H^{\ell^1}_\bullet(M,\mathbb{R})$ of $M$ is defined as the homology of $\ell^1$--chain complex of $M$,
$$H^{\ell^1}_\bullet(M,\mathbb{R}) = H_\bullet ( C_\bullet^{\ell^1}(M,\mathbb{R})).$$

The natural inclusion $C_\bullet(M,\mathbb{R}) \hookrightarrow C_\bullet^{\ell^1}(M,\mathbb{R})$ induces a comparison map $H_\bullet(M,\mathbb{R}) \rightarrow H_\bullet^{\ell^1}(M,\mathbb{R})$. Note that this map is an isometric inclusion because $C_\bullet(M,\mathbb{R})$ is a dense subcomplex of $C_\bullet^{\ell^1}(M,\mathbb{R})$ \cite[Proposition 2.4]{Lo08}.

Similarly, inclusions $C_\bullet(M,\mathbb{R}) \subset C_\bullet^\mathrm{lf}(M,\mathbb{R}) \cap C_\bullet^{\ell^1}(M,\mathbb{R}) \subset C_\bullet^{\ell^1}(M,\mathbb{R})$ imply that the middle complex is dense in $C_\bullet^{\ell^1}(M,\mathbb{R})$. Hence, the induced map $H_\bullet (C_\bullet^\mathrm{lf}(M,\mathbb{R}) \cap C_\bullet^{\ell^1}(M,\mathbb{R})) \rightarrow H^{\ell^1}_\bullet(M,\mathbb{R})$ is an isometric inclusion.
From this point of view, the simplicial volume of $M$ can be computed in terms of the $\ell^1$--homology of $M$ as follows:
$$ \| M \| = \inf \{ \| \alpha \|_1 \text{ }|\text{ }\alpha \in [M]^{\ell^1} \subset H^{\ell^1}_n(M,\mathbb{R}) \},$$
where $[M]^{\ell^1}$ is the set of
all $\ell^1$--homology classes that are represented by at least one locally finite fundamental cycle.

In a similar way, the geometric simplicial volume of $M$ is computed by
$$ \| M \|_\mathrm{Lip} = \inf \{ \| \alpha \|_1 \text{ }|\text{ }\alpha \in [M]^{\ell^1}_\mathrm{Lip} \subset H^{\ell^1}_n(M,\mathbb{R}) \},$$
where $[M]^{\ell^1}_\mathrm{Lip}$ is the set of
all $\ell^1$--homology classes that are represented by at least one locally finite fundamental cycle with finite Lipschitz constant. We refer the reader to \cite[Section 6]{Lo08} for  more detailed explanations.

\subsection{Continuous bounded cohomology}

Let $G$ be a topological group. Consider the continuous cocomplex $C^\bullet_c(G,\mathbb{R})$ with the homogeneous coboundary operator, where $$C^k_c(G,\mathbb{R})=\{ f \co G^{k+1} \rightarrow \mathbb{R}\text{ }|\text{ }f \text{ is continuous} \}.$$
The action of $G$ on $C^k_c(G,\mathbb{R})$ is given by
$$(g\cdot f)(g_0,\ldots,g_k)=f(g^{-1}g_0,\ldots,g^{-1}g_k).$$
The continuous cohomology $H^\bullet_c(G,\mathbb{R})$ of $G$ with trivial coefficients is defined as the cohomology of the $G$--invariant continuous cocomplex $C^\bullet_c(G,\mathbb{R})^G$.

For a cochain $f\co G^{k+1} \rightarrow \mathbb{R}$, define its sup norm by
$$\|f\|_\infty = \sup \{ |f(g_0,\ldots,g_k)|\text{ }|\text{ } (g_0,\ldots,g_k)\in G^{k+1}\}.$$
The sup norm turns $C^\bullet_c(G,\mathbb{R})$ into normed real vector spaces. The continuous bounded cohomology $H^\bullet_{c,b}(G,\mathbb{R})$ of $G$ is defined as the cohomology of the subcocomplex $C^\bullet_{c,b}(G,\mathbb{R})^G$ of
$G$--invariant continuous bounded cochains in $C^\bullet_c(G,\mathbb{R})^G$.
The inclusion of $C^\bullet_{c,b}(G,\mathbb{R})^G \subset C^\bullet_c(G,\mathbb{R})^G$ induces a comparison map $c \co H^\bullet_{c,b}(G,\mathbb{R}) \rightarrow H^\bullet_c(G,\mathbb{R})$. The sup norm induces seminorms on both $H^\bullet_c(G,\mathbb{R})$ and $H^\bullet_{c,b}(G,\mathbb{R})$, denoted by $\| \cdot \|_\infty$. Note that for $\beta \in H^k_c(G,\mathbb{R})$,
$$ \|\beta \|_\infty = \inf \{ \| \beta_b \|_\infty \text{ }|\text{ } \beta_b \in H^k_{c,b}(G,\mathbb{R}) \text{ and } c(\beta_b)=\beta \}.$$

For a connected semisimple Lie group $G$ with trivial center and no compact factors, the continuous cohomology $H^\bullet_c(G,\mathbb{R})$ is isomorphic to the set of $G$--invariant differential forms on the associated symmetric space $X$ according to the Van Est isomorphism. In particular,
the continuous cohomology of $G$ in the top degree is generated by the $G$--invariant volume form $\omega$ on $X$.

Let $\Gamma_0$ be a uniform lattice in $G$ and $M=\Gamma_0 \backslash X$.
Bucher-Karlsson \cite{Bu08} reformulates a proof of Gromov's proportionality principle in the language of continuous bounded cohomology and moreover, shows that $$\frac{\| M \|}{\mathrm{Vol}(M)}=\frac{1}{\| \omega \|_\infty}.$$
It is easy to see that $\|M\|_\mathrm{Lip}=\|M\|$ because $M$ is closed. Let $\Gamma$ be an arbitrary lattice in $G$ and $N=\Gamma \backslash X$. It follows from Gromov's proportionality principle that
\begin{eqnarray}\label{eqn:2.1}
\frac{\| N \|_\mathrm{Lip}}{\mathrm{Vol}(N)}=\frac{\| M \|_\mathrm{Lip}}{\mathrm{Vol}(M)}=\frac{\| M \|}{\mathrm{Vol}(M)}=\frac{1}{\| \omega \|_\infty}.
\end{eqnarray}
Note that the proportionality principle fails in general for the ordinary simplicial volume.

\section{Volume invariant}\label{sec:3}

In this section, we define a new invariant $\mathrm{Vol}(\rho)$ and explore its properties.
Throughout the paper, $G$ denotes a connected semisimple Lie group with trivial center and no compact factors,
and $\Gamma$ denotes a lattice in $G$. As usual, $X$ denotes the associated symmetric $n$--space and $M$ denotes the locally symmetric space $\Gamma\backslash X$. The symbol $\omega$ denotes the $G$--invariant volume form on $X$.

\subsection{Volume invariant}\label{sec:3.1}

Let $\rho \co \Gamma \rightarrow G$ be a representation. Then, $\rho$ induces canonical pullback map
$\rho^*_c \co H^\bullet_c(G,\mathbb{R}) \rightarrow H^\bullet(\Gamma,\mathbb{R})$ in continuous cohomology. This canonical pullback map is realized on the level of cocomplex as follows:
For a continuous map $f \co G^{k+1}\rightarrow \mathbb{R}$, define a map $\rho^*(f) \co \Gamma^{k+1} \rightarrow \mathbb{R}$ by
$$\rho^*(f)(\gamma_0,\ldots,\gamma_k)=f(\rho(\gamma_0),\ldots,\rho(\gamma_k)),$$
for $(\gamma_0,\ldots,\gamma_k) \in \Gamma^{k+1}$.
This defines a chain map $\rho^* \co C^\bullet_c(G,\mathbb{R}) \rightarrow C^\bullet(\Gamma,\mathbb{R})$.
Moreover, $\rho^*$ maps $G$--invariant cochains to $\Gamma$--invariant cochains and hence, it induces a homomorphism $\rho^*_c \co H^\bullet_c(G,\mathbb{R}) \rightarrow H^\bullet(\Gamma,\mathbb{R})$ in continuous cohomology. In the same manner, $\rho$ induces a homomorphism $\rho^*_b \co H^\bullet_{c,b}(G,\mathbb{R}) \rightarrow H^\bullet_b(\Gamma,\mathbb{R})$ in continuous bounded cohomology.

For a connected semisimple Lie group $G$ with trivial center and no compact factors, it is well known that the $G$--invariant volume form $\omega \in H^n_c(G,\mathbb{R})$ is bounded. In other words, there exists a continuous bounded cohomology class $\omega_b \in H^n_{c,b}(G,\mathbb{R})$ such that $c(\omega_b)=\omega$ for the comparison map $c \co H^n_{c,b}(G,\mathbb{R}) \rightarrow H^n_c(G,\mathbb{R})$.
By pulling back $\omega_b$ by $\rho$, we obtain a bounded cohomology class $\rho^*_b(\omega_b)\in H^n_b(\Gamma,\mathbb{R})$. Subsequently, we identify the bounded cohomology class $\rho^*_b(\omega_b)$ in $H^n_b(\Gamma,\mathbb{R})$ with a bounded cohomology class in $H^n_b(M,\mathbb{R})$ via the canonical isomorphism between $H^\bullet_b(\Gamma,\mathbb{R})$ and $H^\bullet_b(M,\mathbb{R})$ \cite{Gr82}.
Then, the bounded cohomology class $\rho^*_b(\omega_b)$ can be evaluated on $\ell^1$--homology classes in $H^{\ell^1}_n(M,\mathbb{R})$ by the Kronecker products
$$ \langle\cdot ,\cdot \rangle \co H^\bullet_b(M,\mathbb{R}) \otimes H^{\ell^1}_\bullet(M,\mathbb{R}) \rightarrow \mathbb{R}.$$

Now, we define a \emph{volume invariant $\mathrm{Vol}(\rho)$ of $\rho$} by
$$\mathrm{Vol}(\rho) = \inf \{ |\langle \rho^*_b(\omega_b),\alpha \rangle| \text{ }|\text{ }c(\omega_b)=\omega \text{ and } \alpha \in [M]^{\ell^1}_\mathrm{Lip} \}.$$

It is easy to see that the volume invariant $\mathrm{Vol}(\rho)$ is finite since $\omega$ is bounded and the geometric simplicial volume of $M$ is strictly positive. Furthermore, a upper bound on the volume invariant $\mathrm{Vol}(\rho)$ can be obtained from its definition immediately as follows.

\begin{prop}\label{pro:3.1}
Let $\rho \co \Gamma \rightarrow G$ be a representation. Then, the volume invariant $\mathrm{Vol}(\rho)$ of $\rho$ satisfies an inequality $$ \mathrm{Vol}(\rho) \leq \mathrm{Vol}(M).$$
\end{prop}

\begin{proof}
For a continuous cohomology class $\beta \in H^n_c(G,\mathbb{R})$,
$$ \| \beta \|_\infty = \inf \{ \| \beta_b \|_\infty \ | \ c(\beta_b)=\beta \},$$
where $c \co H^n_{c,b}(G,\mathbb{R}) \rightarrow H^n_c(G,\mathbb{R})$ is the comparison map.
From the definition of the volume invariant $\mathrm{Vol}(\rho)$, we have
{\setlength\arraycolsep{2pt}
\begin{eqnarray*}
\mathrm{Vol}(\rho) &=& \inf \{ |\langle \rho^*_b(\omega_b),\alpha \rangle| \ | \ c(\omega_b)=\omega \text{ and } \alpha \in [M]^{\ell^1}_\mathrm{Lip} \} \\
&\leq& \inf \{ \| \rho^*_b(\omega_b) \|_\infty \cdot \| \alpha \|_1 \ | \ c(\omega_b)=\omega \text{ and } \alpha \in [M]^{\ell^1}_\mathrm{Lip} \} \\
&\leq& \inf \{ \|\omega_b\|_\infty \ | \ c(\omega_b)=\omega \} \cdot \inf \{ \|\alpha \|_1 \ | \ \alpha\in [M]^{\ell^1}_\mathrm{Lip} \} \\
&=& \|\omega\|_\infty \cdot \| M \|_\mathrm{Lip} \\
&=& \mathrm{Vol}(M).
\end{eqnarray*}}
The last equation comes from Equation (\ref{eqn:2.1}).
\end{proof}

\begin{rem}
If we define the volume invariant $\mathrm{Vol}(\rho)$ via $[M]^{\ell^1}$ instead of $[M]^{\ell^1}_\mathrm{Lip}$, we obtain the following inequality in a similar way as above
$$\mathrm{Vol}(\rho) \leq \|\omega\|_\infty \cdot \| M \|.$$
If $\Gamma$ is a lattice of $\mathbb{Q}$--rank at least $3$, it is known that $\| M \|=0$ \cite{LS09}. This implies that $\mathrm{Vol}(\rho)=0$ for all representations $\rho \co \Gamma \rightarrow G$.
Then, this volume invariant cannot detect discrete, faithful representations. This is the reason why we use the notion of the geometric simplicial volume of $M$ to define the volume invariant $\mathrm{Vol}(\rho)$ instead of the ordinary simplicial volume of $M$.
\end{rem}

\subsection{Volume invariant and $\rho$--equivariant map}

Goldman \cite{Go92} defined the volume invariant $\upsilon(\rho)$ by using a section $s \co M \rightarrow E_\rho$. Indeed, a section $s \co M \rightarrow E_\rho$ corresponds to a $\rho$--equivariant map $s \co X \rightarrow X$.
In a similar way, the volume invariant $\mathrm{Vol}(\rho)$ can be reformulated in terms of $\rho$--equivariant map. In this section, we devote ourselves to explaining this and verifying $\mathrm{Vol}(\rho)=|\upsilon(\rho)|$ for representations $\rho \co \Gamma \rightarrow G$ of uniform lattices $\Gamma$.

First, we describe another useful cocomplexes for both continuous and continuous bounded cohomology of $G$. For a nonnegative integer $k$, define
$$C^k_c(X,\mathbb{R}) = \{ f \co X^{k+1} \rightarrow \mathbb{R} \ | \ f \text{ is continuous} \}.$$
Consider the sup norm $\| \cdot \|_\infty$ on $C^k_c(X,\mathbb{R})$ defined by
$$\| f \|_\infty = \sup \{ |f(x_0,\ldots,x_k)| \ | \ (x_0,\ldots,x_k)\in X^{k+1} \}.$$
Let $C^k_{c,b}(X,\mathbb{R})$ be the subspace consisting of continuous bounded $k$--cochains.
Then, $C^\bullet_c(X,\mathbb{R})$ with the homogeneous coboundary operator becomes a cochain complex. Moreover, the homogeneous coboundary operator on $C^\bullet_c(X,\mathbb{R})$ restricts to $C^\bullet_{c,b}(X,\mathbb{R})$.
The $G$--action on $C^\bullet_c(X,\mathbb{R})$ is defined analogously to the one on $C^\bullet_c(G,\mathbb{R})$.

It is a standard fact that the continuous cohomology $H^\bullet_c(G,\mathbb{R})$ of $G$ is isometrically isomorphic to the cohomology of the cocomplex $C^\bullet_c(X,\mathbb{R})^G$. For a proof, see \cite[Chapter 3]{Gu80}. The continuous bounded cohomology $H^\bullet_{c,b}(G,\mathbb{R})$ of $G$  is isometrically isomorphic to the cohomology of the subcocomplex $C^\bullet_{c,b}(X,\mathbb{R})^G$ of $C^\bullet_c(X,\mathbb{R})^G$.
The comparison map $c \co H^\bullet_{c,b}(G,\mathbb{R})\rightarrow H^\bullet_c(G,\mathbb{R})$ is induced by the natural inclusion $C^\bullet_{c,b}(X,\mathbb{R})^G \subset C^\bullet_c(X,\mathbb{R})^G$.
Furthermore, both $H^\bullet(\Gamma,\mathbb{R})$ and $H^\bullet_b(\Gamma,\mathbb{R})$ are isometrically isomorphic to the cohomologies of cocomplexes $C^\bullet_c(X,\mathbb{R})^\Gamma$ and  $C^\bullet_{c,b}(X,\mathbb{R})^\Gamma$ respectively. See \cite[Corollary 7.4.10]{Mo01} for a detailed proof.

We describe here an explicit map on the level of cocomplex which induces an isometric isomorphism between $H^\bullet(C^\bullet_c(X,\mathbb{R})^G)$ and $H^\bullet_c(G,\mathbb{R})$. Let us fix a base point $o\in X$.
Define a map $\phi_o \co C^k_c(X,\mathbb{R}) \rightarrow C^k_c(G,\mathbb{R})$ by
$$\phi_o(f)(g_0,\ldots,g_k)=f(g_0\cdot o,\ldots, g_k\cdot o).$$
The map $\phi_o$ is a $G$-morphism between two cocomplexes and restricts to the subcocomplexes of continuous bounded cochains. Then, $\phi_o$ induces an isometric isomorphism $\phi^G_c \co H^\bullet(C^\bullet_c(X,\mathbb{R})^G) \rightarrow H^\bullet_c(G,\mathbb{R})$ in continuous cohomology.
Note that $\phi^G_c$ is independent of the choice of the base point $o\in X$ even though $\phi_o$ depends on $o\in X$. Hence, we denote the induced map in continuous cohomology by $\phi^G_c$ without the subscript ``o". In a similar way, the map $\phi_o$ induces isometric isomorphisms, $\phi^G_b \co H^\bullet(C^\bullet_{c,b}(X,\mathbb{R})^G) \rightarrow H^\bullet_{c,b}(G,\mathbb{R})$ and
$\phi^\Gamma_b \co H^\bullet(C^\bullet_{c,b}(X,\mathbb{R})^\Gamma) \rightarrow H^\bullet_b(\Gamma,\mathbb{R})$.

Let $s \co X \rightarrow X$ be a $\rho$--equivariant continuous map for a representation $\rho \co \Gamma \rightarrow G$. Then, $s$ induces a map  $s^* \co C^k_c(X,\mathbb{R}) \rightarrow C^k_c(X,\mathbb{R})$ defined by $$s^*(f)(x_0,\ldots,x_k)=f(s(x_0),\ldots,s(x_k)),$$
for a cochain $f$ in $C^k_c(X,\mathbb{R})$.
Due to the $\rho$--equivariance and continuity of $s \co X\rightarrow X$, it follows that $s^*$ maps $G$--invariant continuous (bounded) cochains to $\Gamma$--invariant continuous (bounded) cochains. Hence, $s^*$ induces  homomorphisms $s^*_c \co H^\bullet( C^\bullet_c (X,\mathbb{R})^G) \rightarrow  H^\bullet( C^\bullet_c (X,\mathbb{R})^\Gamma)$ in continuous cohomology and $s^*_b \co H^\bullet( C^\bullet_{c,b} (X,\mathbb{R})^G) \rightarrow  H^\bullet( C^\bullet_{c,b} (X,\mathbb{R})^\Gamma)$ in continuous bounded cohomology. Now, consider the following diagram:
$$ \xymatrixcolsep{4pc}\xymatrix{
C^\bullet_c(X,\mathbb{R})^G \ar[r]^-{\phi_o} &
C^\bullet_c(G,\mathbb{R})^G \\
C^\bullet_{c,b}(X,\mathbb{R})^G \ar[r]^-{\phi_o} \ar[d]_-{s^*} \ar[u]^-{i} &
C^\bullet_{c,b}(G,\mathbb{R})^G \ar[d]^-{\rho^*}  \ar[u]_-{i} \\
C^\bullet_{c,b}(X,\mathbb{R})^\Gamma \ar[r]^-{\phi_o} &
C^\bullet_b(\Gamma,\mathbb{R})^\Gamma.
}$$

In this diagram, it is clear that the upper diagram commutes.
On the other hand, the lower diagram does not commute. However, one can notice that it commutes in cohomology as follows.
Let $f\in C^k_{c,b}(X,\mathbb{R})^G$ be a $G$--invariant continuous bounded cocycle. Define $b \in C^{k-1}_b(\Gamma,\mathbb{R})$ by
$$b(\gamma_0,\ldots,\gamma_{k-1})=\sum_{i=0}^{k-1} (-1)^i f(\rho(\gamma_0) \cdot o, \ldots, \rho(\gamma_i) \cdot o, \rho(\gamma_i) \cdot s(o),\ldots,\rho(\gamma_{k-1})\cdot s(o)).$$
Then, $b$ is a $\Gamma$--invariant bounded cochain since $f$ is a $G$--invariant continuous bounded cocycle.
Also, it is a straightforward computation that
$$(\rho^* \circ \phi_o - \phi_o \circ s^*)(f)(\gamma_0,\ldots,\gamma_k) = \delta b (\gamma_0,\ldots,\gamma_k).$$ This implies that the lower diagram commutes in the cohomology level and hence, we have the following commutative diagram:
$$ \xymatrixcolsep{4pc}\xymatrix{
H^\bullet(C^\bullet_c(X,\mathbb{R})^G) \ar[r]^-{\phi^G_c}_-{\cong} &
H^\bullet_c(G,\mathbb{R}) \\
H^\bullet(C^\bullet_{c,b}(X,\mathbb{R})^G) \ar[r]^-{\phi^G_b}_-{\cong} \ar[d]_-{s^*_b} \ar[u]^-{c} &
H^\bullet_{c,b}(G,\mathbb{R}) \ar[d]^-{\rho^*_b}  \ar[u]_-{c} \\
H^\bullet(C^\bullet_{c,b}(X,\mathbb{R})^\Gamma) \ar[r]^-{\phi^\Gamma_b}_-{\cong} &
H^\bullet_b(\Gamma,\mathbb{R})
}$$

Each cohomology class in $H^\bullet_c(G,\br)$, $H^\bullet_{c,b}(G,\br)$ and $H^\bullet_b(\Gamma,\br)$ is canonically identified with a cohomology class in $H^\bullet(C^\bullet_c(X,\mathbb{R})^G)$, $H^\bullet(C^\bullet_{c,b}(X,\mathbb{R})^G)$ and $H^\bullet(C^\bullet_{c,b}(X,\mathbb{R})^\Gamma)$ via the isomorphisms induced by $\phi_o$ respectively.

Let $\omega_b$ be a continuous bounded cohomology class in $H^n_{c,b}(G,\br)$ representing the $G$--invariant volume form $\omega \in H^n_c(G,\br)$. We use the same notations $\omega$ and $\omega_b$ for the cohomology class in $H^\bullet(C^\bullet_c(X,\mathbb{R})^G)$ and $H^\bullet(C^\bullet_{c,b}(X,\mathbb{R})^G)$ identified with $\omega \in H^n_c(G,\mathbb{R})$ and $\omega_b \in H^n_{c,b}(G,\mathbb{R})$ via $\phi^G_c$ and $\phi^G_b$, respectively.
 %Then, it is clear that $c(s^*_b(\omega_b))=s^*_c \omega$.

Noting that the cohomologies $H^\bullet(C^\bullet_{c,b}(X,\mathbb{R})^\Gamma)$ and $H^\bullet_b(\Gamma,\mathbb{R})$ are canonically identified with the bounded cohomology $H^\bullet_b(M,\mathbb{R})$, one can conclude that $s^*_b(\omega_b) = \rho^*_b(\omega_b)$ in $H^n_b(M,\br)$ via the canonical isomorphisms. Hence,
$$\{ s^*_b(\omega_b) \in H^n_b(M,\br) \ | \ c(\omega_b) =\omega \} =\{ \rho^*_b(\omega_b) \in H^n_b(M,\br) \ | \ c(\omega_b) =\omega \}.$$

Therefore, the volume invariant $\mathrm{Vol}(\rho)$ can be reformulated in terms of $\rho$--equivariant map as follows:
$$\mathrm{Vol}(\rho)= \inf \{|\langle s^*_b(\omega_b),\alpha \rangle| \ | \ c(\omega_b)=\omega \text{ and } \alpha \in [M]^{\ell^1}_\mathrm{Lip} \}.$$
Note that the above reformulation of the volume invariant $\mathrm{Vol}(\rho)$ is independent of the choice of $\rho$--equivariant map $s \co X \rightarrow X$
as observed.

To define the volume invariant $\upsilon(\rho)$, Goldman \cite{Go92} uses a smooth section of the associated bundle. The reformulation of the volume invariant $\mathrm{Vol}(\rho)$ in terms of $\rho$--equivariant map makes it possible to verify the relation between two invariants $\upsilon(\rho)$ and $\mathrm{Vol}(\rho)$.

\begin{lemma}\label{lem:3.3}
Let $\Gamma$ be a uniform lattice in $G$ and $\rho \co \Gamma \rightarrow G$ be a representation. Then,
$$\mathrm{Vol}(\rho) = |\upsilon (\rho)| = \left| \int_M s^* \omega \right|,$$
where $s \co M \rightarrow E_\rho$ is a smooth section of the associated bundle $E_\rho$.
\end{lemma}

\begin{proof}
A section $s \co M \rightarrow E_\rho$ corresponds to a $\rho$--equivariant map $X\rightarrow X$, denoted by $s \co X \rightarrow X$. Since $M=\Gamma\backslash X$ is a closed manifold, the set $[M]^{\ell^1}_\mathrm{Lip}$ contains exactly one element, namely, the class $i_*[M]$, where
$[M]$ is the fundamental class of $M$, and
$i_* \co H_n(M,\mathbb{R}) \rightarrow H^{\ell^1}_n(M,\mathbb{R})$ is the map induced by the inclusion
$C_\bullet(M,\mathbb{R}) \subset C_\bullet^{\ell^1}(M,\mathbb{R})$.
Hence, the volume invariant $\mathrm{Vol}(\rho)$ is computed by
{\setlength\arraycolsep{2pt}
\begin{eqnarray*}
\mathrm{Vol}(\rho) &=& \inf \{ |\langle s^*_b(\omega_b),\alpha \rangle | \ | \ c(\omega_b)=\omega \text{ and } \alpha \in [M]^{\ell^1}_\mathrm{Lip}  \} \\
&=& \inf \{ |\langle s^*_b(\omega_b),i_*[M] \rangle | \ | \ c(\omega_b)=\omega \}.
\end{eqnarray*}}

Considering the following commutative diagram,
$$ \xymatrixcolsep{4pc}\xymatrix{
H^n_{c,b}(G,\mathbb{R}) \ar[r]^-{c} \ar[d]_-{s^*_b} &
H^n_c(G,\mathbb{R}) \ar[d]^-{s^*_c} \\
H^n_b(\Gamma,\mathbb{R}) \ar[r]^-{c} &
H^n(\Gamma,\mathbb{R})
}$$
we have $c(s^*_b(\omega_b))=s^*_c(c(\omega_b))=s^*_c\omega$. Note that $s^*_c\omega$ is represented by a $\Gamma$--invariant cocycle $s^*f$ where $f \co X^{n+1} \rightarrow \mathbb{R}$ is the $G$--invariant cocycle representing $\omega$,  which is defined by
$$ f(x_0,\ldots,x_n) = \int_{[x_0,\ldots,x_n]} \omega.$$

Also, one can consider another $\Gamma$--invariant cocycle $h \co X^{n+1} \rightarrow \mathbb{R}$ defined by
$$h(x_0,\ldots,x_n) = \int_{[x_0,\ldots,x_n]} s^*\omega.$$
Here, $s^*\omega$ is the pull-back of the $G$--invariant volume form $\omega$ by $s \co X \rightarrow X$.
It is easy to see that $h$ also represents the continuous cohomology class $s^*_c \omega$ because the geodesic straightening map is chain homotopic to the identity.

Let $c$ be a fundamental cycle representing $[M]$.
Since $h$ represents the cohomology class $s^*_c\omega$ in $H^n(\Gamma,\mathbb{R}) \cong H^n(M,\mathbb{R})$, we have
$$ \langle s^*_b(\omega_b), i_*[M] \rangle =\langle s^*_c \omega, [M] \rangle = \langle h, c \rangle=\int_M s^*\omega$$
for any $\omega_b \in c^{-1}(\omega)$. The last equation follows from the de Rham theorem. This completes the proof.
\end{proof}

Goldman proves that $\upsilon (\rho)$ exactly characterizes discrete, faithful representations of $\Gamma$ into $G$ for the case that $G$ is either a connected semisimple Lie group of higher rank or $\text{SO}(n,1)$. This implies that $\mathrm{Vol}(\rho)$ does so by Lemma \ref{lem:3.3}.

\section{Semisimple Lie groups of higher rank}\label{sec:4}

In this section, we prove Theorem \ref{thm:1.2} for the case that $G$ is a semisimple Lie group of higher rank.
Recall the restriction maps $$res_c \co H^\bullet_c(G,\mathbb{R}) \rightarrow H^\bullet(\Gamma,\mathbb{R}) \text{ and } res_b \co H^\bullet_{c,b}(G,\mathbb{R}) \rightarrow H^\bullet_b(\Gamma,\mathbb{R}),$$ induced from the inclusions $C^\bullet_c(X,\mathbb{R})^G \subset C^\bullet_c(X,\mathbb{R})^\Gamma$ and $C^\bullet_{c,b}(X,\mathbb{R})^G \subset C^\bullet_{c,b}(X,\mathbb{R})^\Gamma$ respectively.
Note that $res_b$ is an isometric embedding because $\Gamma$ is a lattice in $G$.
We first observe that $$\langle res_b(\omega_b), \alpha \rangle = \mathrm{Vol}(M)$$ for all $\omega_b \in c^{-1}(\omega)$ and all $\alpha \in [M]^{\ell^1}_\mathrm{Lip}$. To verify this, we need to prove the existence of the geodesic straightening map on the locally finite chain complex with finite Lipschitz constant.

The geodesic straightening map on the singular chain complex of a nonpositively curved manifold is introduced by Thurston \cite[Section 6.1]{Th78}.
Let $X$ be a simply connected, complete Riemannian manifold with nonpositive sectional curvature.
A geodesic simplex is defined inductively as follows: Let $x_0,\ldots,x_k \in X$. First, the geodesic $0$--simplex $[x_0]$ is the point $x_0 \in X$ and the geodesic $1$--simplex $[x_0,x_1]$ is the unique geodesic from $x_1$ to $x_0$. In general, the geodesic $k$--simplex $[x_0,\ldots,x_k]$ is the geodesic cone over $[x_0,\ldots,x_{k-1}]$ with the top point $x_k$.

Let $M$ be a connected, complete Riemannian manifold with nonpositive sectional curvature. Then, \emph{the geodesic straightening map} $str \co C_\bullet (M,\mathbb{R}) \rightarrow C_\bullet (M,\mathbb{R})$ is defined by
$$ str(\sigma) = \pi_M \circ [\tilde{\sigma}(e_0),\ldots,\tilde{\sigma}(e_k)],$$
for a singular $k$--simplex $\sigma \co \Delta^k \rightarrow M$ where $\pi_M \co \widetilde{M} \rightarrow M$ is the universal covering map, $e_0,\ldots,e_k$ are the vertices of the standard $k$--simplex $\Delta^k$, and $\tilde{\sigma}$ is a lift of $\sigma$ to the universal cover $\widetilde{M}$.

\begin{prop}\label{pro:4.1}
Let $M$ be a connected, complete, locally symmetric space of noncompact type. Then, geodesic straightening map is well-defined on $C^\mathrm{lf,Lip}_\bullet(M,\mathbb{R})$ and moreover, it is chain homotopic to the identity.
\end{prop}

\begin{proof}
Let $A \in S^\mathrm{lf,Lip}_k(M)$ for a nonnegative integer $k$. This means that any compact subset of $M$ intersects the image of only finitely many elements of $A$ and there exists a constant $C_A>0$ such that $\mathrm{Lip}(\sigma)<C_A$ for all $\sigma\in A$.

Let $str \co C_\bullet(M,\mathbb{R}) \rightarrow C_\bullet(M,\mathbb{R})$ denote the geodesic straightening map. Define $str(A)$ by
$$str(A)= \{ str(\sigma) \text{ }|\text{ } \sigma \in A \}.$$
To show that geodesic straightening map is well defined on $C^\mathrm{lf,Lip}_\bullet(M,\mathbb{R})$, it is sufficient to show that $str(A) \in S^\mathrm{lf,Lip}_k(M)$.

We first claim that $str(A)$ has finite Lipschitz constant.
Let $\mathrm{Diam}(\sigma)$ denote the diameter of $\sigma(\Delta^k)$ for a singular simplex $\sigma \co \Delta^k \rightarrow M$. For all $\sigma \in A$,
$$\mathrm{Diam}(\sigma) \leq C_A \cdot \mathrm{Diam}(\Delta^k)$$
since $\sigma \co \Delta^k \rightarrow M$ has Lipschitz constant $C_A$.
Hence, each $\sigma \in A$ is contained in a closed ball of diameter $D_A= C_A \cdot \mathrm{Diam}(\Delta^k)$ in $M$. Because every closed ball in $X$ is geodesically convex,
both $\sigma$ and $str(\sigma)$ are contained in the same closed ball of diameter $D_A$ for every $\sigma \in A$. This implies that $\mathrm{Diam}(str(\sigma))<D_A$ for all $\sigma \in A$.

For every $D>0$ and $k\in \mathbb{N}$, there is $L>0$ such that every geodesic $k$--simplex $\tau$ of diameter less than $D$ satisfies $\| T_x\tau \|<L$ for every $x \in \Delta^k$ \cite[Proposition 2.4]{LS09}. Hence,
there exists $L_A>0$ such that $\mathrm{Lip}(str(\sigma)) < L_A $ for all $\sigma \in A$, that is, $str(A)$ has finite Lipschitz constant $L_A$.

Next, to verify that $str(A)$ has locally finite support, we need to show that every compact subset of $M$ intersects the image of only finitely many elements of $str(A)$. Let $K$ be a compact subset of $M$ and $\mathcal{N}_{D_A}(K)$ be the $D_A$--neighborhood of $K$.
Suppose $\overline{\mathcal{N}_{D_A}(K)}\cap \sigma= \emptyset$ for some $\sigma \in A$.
As observed above, both $\sigma$ and $str(\sigma)$ are contained in a closed ball $B_\sigma$ of diameter $D_A$.
It is obvious that $B_\sigma \cap (M-\overline{\mathcal{N}_{D_A}(K)}) \neq \emptyset$ because of $\sigma \subset B_\sigma$.
Then, $B_\sigma$ can never touch $K$, which implies $str(\sigma) \cap K = \emptyset$.
Thus, $K$ can intersect the image of $str(\sigma)$ only for $\sigma \in A$ with $\overline{\mathcal{N}_{D_A}(K)}\cap \sigma \neq \emptyset$.
There exist finitely many such elements of $A$ since $\overline{\mathcal{N}_{D_A}(K)}$ is the compact subset of $M$, and $A$ has locally finite support. Finally, we can conclude that $str(A)$ is a locally finite subset of $S_k(M)$ with finite Lipschitz constant, that is, $str(A) \in S^\mathrm{lf,Lip}_k(M)$.

From the above observation, we have a well-defined map $$str^\mathrm{lf} \co C^\mathrm{lf,Lip}_\bullet(M,\mathbb{R})\rightarrow C^\mathrm{lf,Lip}_\bullet(M,\mathbb{R})$$ extending the geodesic straightening map $str \co C_\bullet (M,\mathbb{R}) \rightarrow C_\bullet (M,\mathbb{R})$.
It is obvious that $str^\mathrm{lf}$ is a chain map.

Now, to construct a chain homotopy from $str^\mathrm{lf}$ to the identity, recall
the chain homotopy $H_\bullet \co C_\bullet(M,\mathbb{R}) \rightarrow C_{\bullet+1}(M,\mathbb{R})$ from the geodesic straightening map $str$ to the identity.
Let $G_\sigma \co \Delta^k \times [0,1] \rightarrow M$ be the canonical straight line homotopy from $\sigma$ to $str(\sigma)$ for a singular $k$--simplex $\sigma$ in $M$. Let $\{e_0,\ldots,e_k \}$ denote the set of vertices in $\Delta^k$ for each $k$. The chain homotopy $H_k \co C_k(M,\mathbb{R})\rightarrow C_{k+1}(M,\mathbb{R})$ is defined by
$$H_k(\sigma) = \sum_{i=0}^k (-1)^i G_\sigma \circ \eta_i,$$
where $\eta_i \co \Delta^{k+1} \rightarrow \Delta^k \times [0,1]$ is the affine map that maps $e_0,\ldots,e_{k+1}$ to $(e_0,0),\ldots,(e_i,0),(e_i,1),\ldots,(e_k,1)$ for $i=0,\ldots,k$.

Let $c=\sum_{\sigma \in A} a_\sigma \sigma$ be a $k$--chain in $C_k^\mathrm{lf,Lip}(M,\mathbb{R})$ for $A\in S_k^\mathrm{lf,Lip}(M)$. Then, as we observed previously, $\mathrm{Lip}(\sigma) < C_A$ and $\mathrm{Lip}(str(\sigma)) < L_A$ for all $\sigma \in A$.
Moreover, the canonical line homotopy $G_\sigma$ from $\sigma$ to $str(\sigma)$ has finite Lipschitz constant that depends only on $C_A$,  $L_A$ by \cite[Proposition 2.1]{LS09}. Noting that the Lipschitz constant of $\eta_i$ is also uniformly bounded from above for all $i=0,\ldots,k$, it follows that the Lipschitz constant of $H_k(\sigma)$ is uniformly bounded from above by a constant depending only on $C_A$, $L_A$ for all $\sigma \in A$. This means that the Lipschitz constant $\mathrm{Lip}(H_k(c))$ of $H_k(c)$ is finite.

To see that $H_k(c)$ has locally finite support, note that if $\sigma$ is contained in a closed ball, then the images of both $str(\sigma)$ and $H_k(\sigma)$ are contained in the same closed ball because every closed ball in $X$ is geodesically convex. As in the proof that $str(A)$ has locally finite support, any compact subset $K$ of $M$ can intersect the image of singular $(k+1)$--simplices occurring in $H_k(\sigma)$ only for $\sigma \in A$ with $\overline{\mathcal{N}_{D_A}(K)}\cap \sigma \neq \emptyset$. The set of such elements of $A$ are finite due to $A\in S^\mathrm{lf,Lip}_k(M)$. Moreover, since $H_k(\sigma)$ is a finite sum of $(k+1)$--simplices, $K$ intersects the image of finitely many $(k+1)$--simplices occurring in $H_k(c)$. This implies that $H_k(c)$ has locally finite support.
Now, we have a well-defined map,
$$H_\bullet^\mathrm{lf} \co C^\mathrm{lf,Lip}_\bullet(M,\mathbb{R}) \rightarrow C^\mathrm{lf,Lip}_{\bullet+1}(M,\mathbb{R}).$$

Since $H_\bullet^\mathrm{lf} \co C^\mathrm{lf,Lip}_\bullet(M,\mathbb{R}) \rightarrow C^\mathrm{lf,Lip}_{\bullet+1}(M,\mathbb{R})$ is the map extending the chain homotopy $H_\bullet \co C_\bullet (M,\mathbb{R}) \rightarrow C_{\bullet +1}(M,\mathbb{R})$ between the geodesic straightening map $str$ and the identity, it clearly satisfies $$\partial \circ H_k^\mathrm{lf} + H_{k-1}^\mathrm{lf} \circ \partial =str^\mathrm{lf} -id.$$
Hence, $H^\mathrm{lf}_\bullet$ is a chain homotopy from $str^\mathrm{lf}$ to the identity.
Therefore, we can conclude that $str^\mathrm{lf} \co C^\mathrm{lf,Lip}_\bullet(M,\mathbb{R}) \rightarrow C^\mathrm{lf,Lip}_\bullet(M,\mathbb{R})$ is a chain homotopic to the identity.
\end{proof}

The existence of the geodesic straightening map on $C^\mathrm{lf,Lip}_\bullet(M,\mathbb{R})$ allows us to get a straight cycle from an arbitrary cycle without changing its homology class. By using the straightening map $str^\mathrm{lf} \co C^\mathrm{lf,Lip}_\bullet(M,\mathbb{R}) \rightarrow  C^\mathrm{lf,Lip}_\bullet(M,\mathbb{R})$, we can prove the following Lemma.

\begin{lemma}\label{lem:4.2}
Let $G$ be a connected semisimple Lie group with trivial center and no compact factors. Let $\Gamma$ be a lattice in $G$. Then,
$$\langle res_b(\omega_b) , \alpha \rangle = \mathrm{Vol}(M)$$
for all $\omega_b \in c^{-1}(\omega)$ and all $\alpha \in [M]^{\ell^1}_\mathrm{Lip}$.
\end{lemma}

\begin{proof}
On the continuous cochain complex $C^\bullet_c(X,\mathbb{R})$, the $G$--invariant volume form $\omega$ is represented by a cocycle $f \co X^{n+1} \rightarrow \mathbb{R}$ defined by
$$f(x_0,\ldots,x_n)=\int_{[x_0,\ldots,x_n]} \omega.$$
Let $f_b \co X^{n+1} \rightarrow \mathbb{R}$ be a cocycle representing $\omega_b \in c^{-1}(\omega)$. Both $f$ and $f_b$ represent the same cohomology class $\omega$ in the continuous cohomology of $G$. Hence, there exists a $G$--invariant continuous cochain $b$ in $C^{n-1}_c(X,\mathbb{R})^G$ such that $$f_b =f +\delta b.$$

Let $c=\sum_{i=1}^\infty a_i \sigma_i$ be a locally finite fundamental $\ell^1$--cycle with finite Lipschitz constant representing $\alpha$.
Due to Proposition \ref{pro:4.1}, $str^\mathrm{lf}(c)$ is also a locally finite fundamental $\ell^1$--cycle with finite Lipschitz constant and represents $\alpha$. By \cite[Proposition 4.4]{LS09}, we have
$$ \langle f, str^\mathrm{lf}(c) \rangle = \mathrm{Vol}(M).$$

Now, we claim that $\langle \delta b, str^\mathrm{lf}(c) \rangle =0$. Let $\sigma_i^j$ denote the $j$-th face of $\sigma$ for $j=0,\ldots,n$. Then, $\partial \sigma_i = \sum_{j=0}^n (-1)^j \cdot \sigma_i^j$ and
\begin{eqnarray}\label{eqn:4.1}
\langle \delta b, str^\mathrm{lf}(c) \rangle = \sum_{i=1}^\infty \sum_{j=0}^n (-1)^j  a_i \cdot \langle b, str(\sigma_i^j) \rangle.
\end{eqnarray}

Since the Lipschitz constant of $str^\mathrm{lf}(c)$ is finite, there exists $R>0$ such that
each $\sigma_i$ is contained in a closed ball with radius $R$ for all $i \in \mathbb{N}$. Fix a closed ball $B$ with radius $R$ in $X$. Then, there exists $g_i \in G$ for each $\sigma_i$ such that $g_i\cdot str(\sigma_i) \subset B$ since $G$ acts transitively on $X$. Due to the $G$--invariance of $b$, we have $\langle b, str(\sigma_i^j) \rangle = \langle b, g_i \cdot str(\sigma_i^j) \rangle$ for all $i\in \mathbb{N}$ and $j=0,\ldots,n$. This implies that $$\langle b, str(\sigma_i^j) \rangle = b(x_0, \ldots,x_{n-1})$$ for some $(x_0,\ldots,x_{n-1})\in B^n$.
Since $b$ is continuous and $B$ is the compact subset of $X$, there exists a upper bound $C>0$ on $\langle b, str(\sigma_i^j) \rangle $ for all $i\in \mathbb{N}$ and $j=0,\ldots,n$. Furthermore, $c$ is a $\ell^1$--cycle and hence,
$$\sum_{i=1}^\infty \sum_{j=0}^n | (-1)^j  a_i \cdot \langle b, str(\sigma_i^j) \rangle | < n C \cdot \sum_{i=1}^\infty |a_i| < \infty.$$

In other words, the series in Equation (\ref{eqn:4.1}) absolutely converges. Thus, all rearrangements of the series in Equation (\ref{eqn:4.1}) converge to the same value. From the cycle condition of $str^\mathrm{lf}(c)$,
there exists a permutation $\tau$ of $\mathbb{N}\times \{0,\ldots,n\}$ such that
$$ \sum_{i=1}^\infty \sum_{j=0}^n (-1)^{\tau(j)}  a_{\tau(i)} \cdot str(\sigma_{\tau(i)}^{\tau(j)})=0.$$
Under this permutation $\tau$, we can conclude that $\langle \delta b, str^\mathrm{lf}(c) \rangle = 0$.
Finally, we have
{\setlength\arraycolsep{2pt}
\begin{eqnarray*}
\langle res_b(\omega_b), \alpha \rangle &=& \langle f+\delta b, str^\mathrm{lf}(c) \rangle \\ &=&
\langle f, str^\mathrm{lf}(c) \rangle + \langle \delta b, str^\mathrm{lf}(c) \rangle \\
&=& \mathrm{Vol}(M)
\end{eqnarray*}}
The second equation is available since all series in the equation absolutely converge.
\end{proof}

\begin{defi}
A representation $\rho \co \Gamma \rightarrow G$ is \emph{maximal} if
$$\mathrm{Vol}(\rho)=\mathrm{Vol}(M).$$
\end{defi}
For reader's convenience, we recall Margulis's normal subgroup theorem \cite{Margulis}.
\begin{thm}Let $G$ be a connected semisimple Lie group with finite center with $\br$-rank $\geq 2$, and let $\Gamma\subset G$ be an irreducible lattice. If $N\subset \Gamma$ is a normal subgroup of $\Gamma$, then either $N$ lies in the center of $G$ or the quotient $\Gamma/N$ is finite.
\end{thm}
\begin{thm}\label{thm:4.4}
Let $G$ be a connected semisimple Lie group of higher rank with trivial center and no compact factors.
Let $\Gamma$ be an irreducible lattice in $G$. Then, a representation $\rho \co \Gamma \rightarrow G$ is maximal if and only if $\rho$ is a discrete, faithful representation.
\end{thm}

\begin{proof}

First, suppose that $\rho$ is discrete and faithful.
Margulis Superrigidity Theorem implies that $\rho$ extends to an automorphism $\tilde{\rho} \co G \rightarrow G$. Then, a representation $\rho \co \Gamma \rightarrow G$ is written as a composition $\rho = \tilde{\rho} \circ i$ where $i \co \Gamma \rightarrow G$ is the natural inclusion of $\Gamma$ into $G$.
The canonical pullback map $\rho^*_b \co H^\bullet_{c,b}(G,\mathbb{R}) \rightarrow H^\bullet_b(\Gamma,\mathbb{R})$ in continuous bounded cohomology is realized as a composition $\rho^*_b=res_b \circ \tilde{\rho}^*_b$,
$$ \xymatrixcolsep{2pc}\xymatrix{
H^\bullet_{c,b}(G,\mathbb{R}) \ar[r]^{\tilde{\rho}^*_b} &
H^\bullet_{c,b}(G,\mathbb{R}) \ar[r]^{res_b} &
H^\bullet_b(\Gamma,\mathbb{R}).
}$$

Since $\tilde{\rho}$ is an automorphism of $G$, it induces an automorphism of the continuous (bounded) cohomology of $G$. In particular, it is easy to see that $\tilde{\rho}^*_c(\omega) = \pm \omega$ in $H^n_c(G,\mathbb{R})$.
Considering the commutative diagram
$$ \xymatrixcolsep{4pc}\xymatrix{
H^n_{c,b}(G,\mathbb{R}) \ar[r]^-{c} \ar[d]_-{\tilde{\rho}^*_b} &
H^n_c(G,\mathbb{R}) \ar[d]^-{\tilde{\rho}^*_c} \\
H^n_{c,b}(G,\mathbb{R}) \ar[r]^-{c} &
H^n_c(G,\mathbb{R})
}$$
the automorphism $\tilde{\rho}^*_b \co H^n_{c,b}(G,\mathbb{R}) \rightarrow H^n_{c,b}(G,\mathbb{R})$ permutes the set of $c^{-1}(\omega)$ up to sign. Hence,
{\setlength\arraycolsep{2pt}
\begin{eqnarray*}
\mathrm{Vol}(\rho) &=& \inf \{ |\langle \rho^*_b (\omega_b),\alpha \rangle| \ | \ c(\omega_b)=\omega \text{ and } \alpha \in [M]^{\ell^1}_\mathrm{Lip} \} \\
&=& \inf \{|\langle res_b( \tilde{\rho}^*_b (\omega_b)),\alpha \rangle | \ | \ c(\omega_b)=\omega \text{ and } \alpha \in [M]^{\ell^1}_\mathrm{Lip} \} \\
&=& \inf \{ |\langle res_b( \omega_b),\alpha \rangle | \ | \ c(\omega_b)=\omega \text{ and } \alpha \in [M]^{\ell^1}_\mathrm{Lip} \}.
\end{eqnarray*}}
According to Lemma \ref{lem:4.2}, $\langle res_b(\omega_b),\alpha \rangle =\mathrm{Vol}(M)$ for all $\omega_b \in c^{-1}(\omega)$ and all $\alpha \in [M]^{\ell^1}_\mathrm{Lip}$. Therefore, $\mathrm{Vol}(\rho)= \mathrm{Vol}(M)$.

Conversely, suppose that $\rho \co \Gamma \rightarrow G$ is not a discrete, faithful representation. If $\rho$ has nontrivial kernel, then $\rho(\Gamma)$ is a finite group by the Margulis's normal subgroup theorem.
If $\rho$ is a nondiscrete, faithful representation, then $\rho(\Gamma)$ is precompact by the Margulis superrigidity theorem. In either case, $\rho(\Gamma)$ is an amenable subgroup of $G$.
Regarding $\rho$ as a composition $\rho = i \circ \rho$,
$$ \xymatrixcolsep{2pc}\xymatrix{
\Gamma \ar[r]^-{\rho} &
\rho(\Gamma) \ar[r]^-{i} &
G
}$$
one can realize $\rho^*_b \co H^\bullet_{c,b}(G,\mathbb{R})\rightarrow H^\bullet_b(\Gamma,\mathbb{R})$ as a composition $\rho^*_b \circ res_b$,
$$ \xymatrixcolsep{2pc}\xymatrix{
H^\bullet_{c,b}(G,\mathbb{R}) \ar[r]^-{res_b} &
H^\bullet_{c,b}(\rho(\Gamma),\mathbb{R}) \ar[r]^-{\rho^*_b} &
H^\bullet_b(\Gamma,\mathbb{R}).
}$$
The continuous bounded cohomology $H^\bullet_{c,b}(\rho(\Gamma),\mathbb{R})$ is trivial because
$\rho(\Gamma)$ is amenable. This implies that $\rho^*_b(\omega_b) = \rho^*_b(res_b(\omega_b))=0$ for all $\omega_b \in c^{-1}(\omega)$. Hence, $\mathrm{Vol}(\rho)=0$.
This completes the proof of this theorem.
\end{proof}

\section{Simplie Lie groups of rank $1$}\label{sec:5}

In this section, we give a proof of Theorem \ref{thm:1.2} for the case that $G$ is a simple Lie group of rank $1$ except for $\text{SO}(2,1)$. The Besson-Courtois-Gallot technique is a central ingredient here.

\begin{defi}
Let $F \co X \rightarrow Y$ be a smooth map between Riemannian manifolds $X$ and $Y$. The $p$--Jacobian $\mathrm{Jac}_pF$ of $F$ is defined by $$\mathrm{Jac}_pF(x) =\sup \| d_xF(u_1) \wedge \cdots \wedge d_xF(u_p) \|,$$ where $\{u_1,\ldots,u_p\}$ varies on the set of orthonormal $p$--frames at $x\in X$.
\end{defi}

Let $X$ and $Y$ be complete, simply connected, Riemannian manifolds.
Suppose that the sectional curvature $K_Y$ on $Y$ satisfies $K_Y \leq -1$.
Let $\Gamma$ and $\Gamma'$ be discrete subgroups of $\text{Isom}(X)$ and $\text{Isom}(Y)$ respectively. For any representation $\rho \co \Gamma \rightarrow \Gamma'=\rho(\Gamma)$,
Besson, Courtois and Gallot show that for all $\epsilon >0$, $p\geq 3$, there exists a $\rho$--equivariant map $F_\epsilon \co X \rightarrow Y$ such that
$$\mathrm{Jac}_p F_\epsilon (x) \leq \left( \frac{\delta(\Gamma)}{p-1}(1+\epsilon) \right)^p,$$
for all $x \in X$ where $\delta(\Gamma)$ is the critical exponent of $\Gamma$.
Furthermore, they show that if $X$ has strictly negative sectional curvature, $\Gamma$ and $\Gamma'$ are convex cocompact and $\rho$ is injective,
then there exists the natural map $F \co X\rightarrow Y$ in \cite[Theorem 1.10]{BCG99}.
Note that one can make use of Besson-Courtois-Gallot's method
if there exists a $\rho$--equivariant measurable map from the visual boundary $\partial X$ of $X$ to $\partial Y$ for a representation $\rho :\Gamma \rightarrow \mathrm{Isom}(Y)$.

\begin{prop}\label{pro:5.2}
Let $G$ and $H$ be connected simple Lie groups of rank $1$ with trivial center and no compact factors.
Let $X$ and $Y$ be the symmetric spaces associated with $G$ and $H$ respectively.
Assume that the symmetric metrics on $X$ and $Y$ are normalized so that their curvatures lie between $-4$ and $-1$.
Let $\Gamma$ be a lattice in $G$ and $\rho \co \Gamma \rightarrow H$ be a representation whose image is nonelementary. Then, there exists a map $F \co X\rightarrow Y$ such that
\begin{itemize}
\item[(1)] $F$ is smooth.
\item[(2)] $F$ is $\rho$--equivariant.
\item[(3)] For all $k\geq 3$, $\mathrm{Jac}_kF(x) \leq (\delta(\Gamma)/(k-1))^k$.
\item[(4)] If $\mathrm{dim}(X) \geq \mathrm{dim}(Y) \geq 3$, then $\mathrm{Jac}_nF(x) \leq (\delta(\Gamma)/(n+d-2))^n$ where $d$ is the real dimension of the field or the ring under consideration for $G$.
Moreover, equality holds for some $x\in X$ if and only if $D_xF$ is a homothety from $T_xX$ to $T_{F(x)}Y$.
\end{itemize}
\end{prop}

\begin{proof}
By the assumption of the sectional curvatures on $X$ and $Y$, the associated symmetric spaces $X$ and $Y$ are $\mathrm{CAT}(-1)$--spaces. Since any lattice in $G$ is a discrete divergence subgroup of $G$, it follows from \cite[Theorem 0.2]{BM96} that there exists the unique $\rho$--equivariant measurable map $\varphi \co \partial X \rightarrow \partial Y$ and it takes almost all its values in the limit set of $\rho(\Gamma)$.

Let $\{ \nu_x \}_{x\in X}$ denote the family of Patterson-Sullivan measures on $\partial X$ for $\Gamma$. Let $\mu_x $ be the pushforward of $\nu_x$ by $\varphi$, that is, $\mu_x = \varphi_* \nu_x$. It can be easily seen that $\{\mu_x \}_{x\in X}$ is $\rho$--equivariant and moreover, the measures $\mu_x$ and $\mu_y$ are in the same measure class for all $x,y \in X$.

We claim that the barycenter of $\mu_x$ is well defined for all $x\in X$. Recall that if $\mu_x$ is not concentrated on two points, then the barycenter of $\mu_x$ is well defined. Assume that $\mu_x$ is concentrated on two points. Let $p$ be one of them. Then, $\mu_x$ must have positive weights on each $\rho(\Gamma)$--orbit of $p$ because $\mu_x$ and $\mu_{\gamma x}=\rho(\gamma)_* \mu_x$ are in the same measure class for all $\gamma \in \Gamma$. However, the set of $\rho(\Gamma)$--orbits of $p$ contains more than two points because $\rho(\Gamma)$ is nonelementary. This contradicts the assumption that $\mu_x$ is concentrated on only two points. Therefore, the claim holds.

As Besson, Courtois and Gallot construct the natural map in \cite{BCG99}, define a map $F \co X\rightarrow Y$ by the composition $bar \circ \varphi_* \circ \mu$ of maps
$$ \xymatrixcolsep{2pc}\xymatrix{
X \ar[r]^-{\mu} &
\mathcal{M}^+ (\partial X) \ar[r]^-{\varphi_*} &
\mathcal{M}^+ (\partial Y) \ar[r]^-{bar} &
Y
}$$
where $\mathcal{M}^+ (\partial X)$ denotes the set of positive Borel measures on $\partial X$.
Then, this map $F$ is a $\rho$--equivariant.
Furthermore, the properties $(1) \sim (4)$ of the natural map $F \co X \rightarrow Y$ can be proved by the same argument as in \cite[Section 2]{BCG99}.
\end{proof}

The map $F \co X\rightarrow Y$ as above is called the \emph{natural map} for a representation $\rho \co\Gamma \rightarrow H$.

\begin{thm}\label{thm:5.3}
Let $G$ be a connected simple Lie group of rank $1$ with trivial center and no compact factors, except for $\mathrm{SO}(2,1)$.
Let $\Gamma$ be a  lattice in $G$. Then, a representation $\rho \co \Gamma \rightarrow G$ is maximal if and only if $\rho$ is a discrete, faithful representation.
\end{thm}

\begin{proof}
Suppose that $\rho \co \Gamma \rightarrow G$ is a discrete, faithful representation.
 Let $X$ be the associated symmetric space of dimension $n$ and $M=\Gamma\backslash X$.
Then, $\rho$ extends to an automorphism $\tilde{\rho} \co G \rightarrow G$ due to the Mostow's rigidity theorem. In a similar argument as in the proof of Theorem \ref{thm:4.4}, we have $\mathrm{Vol}(\rho)=\mathrm{Vol}(M).$

Conversely, we now suppose that $\mathrm{Vol}(\rho)=\mathrm{Vol}(M)$. If $\rho(\Gamma)$ is elementary, then $\rho^*_b(\omega_b)=0$ for all $\omega_b \in c^{-1}(\omega)$ and thus, $\mathrm{Vol}(\rho)=0$. Hence, we can assume that $\rho(\Gamma)$ is nonelementary. Assume that the sectional curvature on $X$ lies between $-4$ and $-1$.
Then, there exists the natural map $F \co X \rightarrow X$ according to Proposition \ref{pro:5.2}.
Because of the critical exponent $\delta(\Gamma)=n+d-2$ for any lattice $\Gamma$ in $G$
where $d$ is the real dimension of the field or the ring under consideration for $G$, we have $$\mathrm{Jac}_nF(x) \leq 1.$$

Define a continuous function $f \co X^{n+1} \rightarrow \mathbb{R}$ by
$$f(x_0,\ldots,x_n) = \int_{[x_0,\ldots,x_n]}\omega.$$
It can be easily seen that $f \co X^{n+1} \rightarrow \mathbb{R}$ is a $G$--invariant continuous bounded cocycle representing the $G$--invariant volume form $\omega \in H^n_c(G,\mathbb{R})$ on $X$.
Hence, $f$ determines a continuous bounded cohomology class $\omega_b \in c^{-1}(\omega)$.
Recall that the $\Gamma$--invariant bounded cocycle $F^* f \co X^{n+1} \rightarrow \mathbb{R}$ is defined by $$F^* f(x_0,\ldots,x_n)=f(F(x_0),\ldots,F(x_n))=\int_{[F(x_0),\ldots,F(x_n)]}\omega.$$

Considering the pullback $F^*\omega$ of the $G$--invariant volume form $\omega$ on $X$ by the natural map $F$, one can define another $\Gamma$--invariant continuous bounded cocycle $h \co X^{n+1}\rightarrow \mathbb{R}$ by $$h(x_0,\ldots,x_n)=\int_{[x_0,\ldots,x_n]}F^*\omega.$$
The change of variables formula implies
$$h(x_0,\ldots,x_n)=\int_{[x_0,\ldots,x_n]}F^*\omega=\int_{F([x_0,\ldots,x_n])} \omega.$$

It is clear that $[F(x_0),\ldots,F(x_n)]= str(F([x_0,\ldots,x_n]))$. From the canonical straight line homotopy $H_\bullet \co C_\bullet(X,\mathbb{R})\rightarrow C_{\bullet+1}(X,\mathbb{R})$ between the geodesic straightening map and the identity, we have
$$ [F(x_0),\ldots,F(x_n)]-F([x_0,\ldots,x_n])=(\partial \circ H_n + H_{n-1} \circ \partial) (F([x_0,\ldots,x_n])).$$

It is a straightforward computation that $ h - F^* f = \delta \eta$ where
$$\eta(x_0,\ldots,x_{n-1})=\int_{H_{n-1}\circ F([x_0,\ldots,x_{n-1}])} \omega .$$
\begin{lemma} If $G$ is not $\mathrm{SO}(3,1)$, $\eta$ is a $\Gamma$-invariant continuous bounded cochain, which implies that $h$ and $F^* f$ represent the same bounded cohomology class $F^*_b(\omega_b)$ in $H^n_b(\Gamma,\mathbb{R})$.
\end{lemma}
\begin{proof}
In the case that $G$ is not $\mathrm{SO}(3,1)$, the associated symmetric space $X$ has dimension at least $4$.
Then, the property ($3$) in Proposition \ref{pro:5.2} shows $$\mathrm{Jac}_{n-1}F(x) \leq \left( \frac{n+d-2}{n-2} \right)^{n-1},$$ for all $x \in X$. Hence, the volume of $F([x_0,\ldots,x_{n-1}])$ has a uniform upper bound. The volume of the straight line homotopy between $F([x_0,\ldots,x_{n-1}])$ and $[F(x_0),\ldots,F(x_{n-1})]$ is uniformly bounded from above
since the volumes of both $F([x_0,\ldots,x_{n-1}])$ and $[F(x_0),\ldots,F(x_{n-1})]$ are uniformly bounded from above and the sectional curvature on $X$ is bounded from above by $-1$. More precisely, one can approximate the straight line homotopy by the union of small cones $C_i$ whose bases are on $[F(x_0),\ldots,F(x_{n-1})]$ and whose apexes are on $F([x_0,\ldots,x_{n-1}])$, and small cones $C_j$ whose bases are on $F([x_0,\ldots,x_{n-1}])$ and whose apexes are on  $[F(x_0),\ldots,F(x_{n-1})]$. On the other hand, it can be shown, see for example \cite{Gr82} (page 19), that
$$\mathrm{Vol}(Cone)\leq (n-1)^{-1}\mathrm{Vol}(Base).$$ This shows that the volume of the straight line homotopy is bounded uniformly by the sum of volumes of
$F([x_0,\ldots,x_{n-1}])$ and $[F(x_0),\ldots,F(x_{n-1})]$.
Thus, $\eta$ is a $\Gamma$-invariant continuous bounded cochain, which implies that $h$ and $F^* f$ represent the same bounded cohomology class $F^*_b(\omega_b)$ in $H^n_b(\Gamma,\mathbb{R})$.
\end{proof}
Let $\alpha \in [M]^{\ell^1}_\text{Lip}$ and $c$ be a locally finite fundamental $\ell^1$--cycle with finite Lipschitz constant representing $\alpha$.
We now assume that $G$ is not $\mathrm{SO}(3,1)$. Maximality condition $\mathrm{Vol}(\rho)=\mathrm{Vol}(M)$ gives us an inequality
\begin{eqnarray}\label{naturalvol}
|\langle F^*_b (\omega_b), \alpha \rangle | = | \langle F^* f, c \rangle | = |\langle h, c \rangle| = \left| \int_M F^*\omega \right|\geq \mathrm{Vol}(M).
\end{eqnarray}
Since $\text{Jac}_nF(x)\leq 1$ almost everywhere, inequality (\ref{naturalvol}) actually implies that
$$ \left| \int_M F^*\omega \right| = \mathrm{Vol}(M),$$
and hence, $\text{Jac}_nF(x)=1$ everywhere. Then, it follows from the property $(4)$ of the natural map in Proposition \ref{pro:5.2} that $F$ is an isometry.
Therefore, $\rho :\Gamma \rightarrow G$ is a discrete, faithful representation.

The theorem for the case $G=\mathrm{SO}(3,1)$ can be covered by the result of Bucher, Burger and Iozzi \cite{BBI}. In their paper \cite{BBI}, an invariant for representations of lattices in $\mathrm{SO}(n,1)$ is defined in the same manner as the invaraint for representations of lattices in $\mathrm{SO}(2,1)$ in \cite{BIW10}. Moreover, they show that the invariant detects discrete, faithful representations for $n \geq 3$. In fact, it is easy to see that the absolute value of the invariant for representations $\rho$ of hyperbolic lattices is equal to the volume invariant $\mathrm{Vol}(\rho)$. This follows from the same argument in the proof of Proposition \ref{prop:6.2}. Hence, the theorem holds for the case $G=\mathrm{SO}(3,1)$. We finally complete the proof.
\end{proof}

From Lemma \ref{lem:3.3}, it is easy to see that Theorem \ref{thm:5.3} covers the remaining cases $\mathrm{SU}(n,1), \mathrm{Sp}(n,1), \mathrm{F}_4^{-20}$ that Goldman's proof in \cite{Go92} did not cover. Hence, we complete the proof of Conjecture \ref{con:1.1}.

\section{$\text{SO}(2,1)$}\label{sec:6}

In this section, we deal with $\text{PU}(1,1)$ instead of $\text{SO}(2,1)$ for convenience.
Let $\Gamma$ be a lattice in $\text{PU}(1,1)$ and $\rho \co \Gamma \rightarrow \text{PU}(1,1)$ be a representation.
The unit ball $\mathbb{D}$ in the complex plane $\mathbb{C}$ is the associated symmetric space and $S=\Gamma\backslash \mathbb{D}$ is a surface of finite topological type with negative Euler number.
If $\Gamma$ is a uniform lattice, then the volume invariant $\mathrm{Vol}(\rho)$ is equal to
$|\upsilon(\rho)|$ as we see this in Lemma \ref{lem:3.3}. Hence, Theorem \ref{thm:1.2} for uniform lattices in $\text{PU}(1,1)$ follows from Goldman's proof. We refer the reader to \cite{Go81} for a detailed proof of this.

From now on, we assume that $\Gamma$ is a nonuniform lattice in $\text{PU}(1,1)$. In this case, Burger, Iozzi and Wienhard define the Toledo invariant as follows.
Let $\Sigma$ be a connected, oriented, compact surface with boundary $\partial \Sigma$ whose interior is homeomorphic to $S$. Let $\rho \co \pi_1(\Sigma) \rightarrow \text{PU}(1,1)$ be a representation.
The second continuous cohomology $H^2_c(\text{PU}(1,1),\mathbb{R})$ of $\text{PU}(1,1)$ is generated by the K\"{a}hler form $\kappa$ on $\mathbb{D}$. There is the unique continuous bounded K\"{a}hler class $\kappa_b \in H^2_{c,b}(\text{PU}(1,1),\mathbb{R})$ since the comparison map $c \co H^\bullet_{c,b}(\mathrm{PU}(1,1),\mathbb{R}) \rightarrow H^\bullet_c(\mathrm{PU}(1,1),\mathbb{R})$ is an isomorphism in degree $2$. By pulling back the bounded K\"{a}hler class $\kappa_b$ via $\rho$, one can obtain a bounded cohomology class $$\rho^*_b(\kappa_b) \in H^2_b(\pi_1(\Sigma),\mathbb{R}) \cong H^2_b(\Sigma,\mathbb{R}).$$

The canonical map $C^\bullet_b(\Sigma, \partial \Sigma, \mathbb{R}) \rightarrow C^\bullet_b(\Sigma,\mathbb{R})$ induces an isomorphism $j \co H^2_b(\Sigma,\partial \Sigma,\mathbb{R}) \rightarrow H^2_b(\Sigma,\mathbb{R})$ in bounded cohomology. The Toledo invariant $\mathrm{T}(\Sigma, \rho)$ of $\rho$ is defined by
$$\text{T}(\Sigma, \rho)= \langle j^{-1}(\rho^*_b(\kappa_b)),[\Sigma,\partial \Sigma]\rangle,$$
where $j^{-1}(\rho^*_b(\kappa_b))$ is considered as an ordinary relative cohomology class and $[\Sigma,\partial \Sigma]$ is the relative fundamental class. Burger, Iozzi and Wienhard obtain a kind of the Milnor inequality
$$ | \mathrm{T}(\Sigma,\rho) | \leq \chi(\Sigma),$$
where $\chi(\Sigma)$ is the Euler number of $\Sigma$. Moreover, they generalize Goldman's characterization of maximal representations for closed surfaces to the cases of surfaces with boundary.

\begin{thm}[Burger, Iozzi and Wienhard]
Let $\Sigma$ be a connected oriented surface with negative Euler number. A representation $\rho \co  \pi_1(\Sigma) \rightarrow \mathrm{PU}(1,1)$ is maximal if and only if it is the holonomy representation of a complete hyperbolic metric on the interior of $\Sigma$.
\end{thm}

In fact, a similar argument holds for a representation of $\pi_1(\Sigma)$ into a Lie group of Hermitian type.
We refer the reader to \cite{BIW10} for more details.

\begin{prop}\label{prop:6.2}
Let $\Gamma$ be a nonuniform lattice in $\mathrm{PU}(1,1)$. Then
$$\mathrm{Vol}(\rho)=2\pi |\mathrm{T}(\Sigma,\rho) |.$$
\end{prop}

\begin{proof}

Let $S=\Gamma\backslash \mathbb{D}$ and $\Sigma$ be the compact surface with boundary whose interior is homeomorphic to $S$. We think of $S$ as the interior of $\Sigma$.
Let $\omega$ be the $\text{PU}(1,1)$--invariant volume form on $\mathbb{D}$. Then, $\omega= 2\pi \kappa$ for the K\"{a}hler form $\kappa$ on $\mathbb{D}$. Hence,
{\setlength\arraycolsep{2pt}
\begin{eqnarray*}
\mathrm{Vol}(\rho)&=& \inf \{ |\langle \rho^*_b (\omega_b), \alpha \rangle| \ | \ c(\omega_b)=\omega \text{ and }\alpha \in [S]^{\ell^1}_\mathrm{Lip} \} \\
&=& 2\pi \cdot \inf \{| \langle \rho^*_b (\kappa_b), \alpha \rangle | \ | \ \alpha \in [S]^{\ell^1}_\mathrm{Lip} \}.
\end{eqnarray*}}

We claim that $\langle \rho^*_b (\kappa_b), \alpha \rangle = \text{T}(\Sigma,\rho)$ for all $\alpha \in [S]^{\ell^1}_\mathrm{Lip}$. Consider a collar neighborhood of $\partial \Sigma$ in $\Sigma$ that is homeomorphic to $\partial \Sigma \times [0,1)$. Let $K$ be the complement of the collar neighborhood of $\partial \Sigma$. Note that $K$ is a compact subsurface with boundary that is a deformation retract of $\Sigma$. Consider the following commutative diagram,
$$ \xymatrixcolsep{2pc}\xymatrix{
C^\bullet_b(S,\mathbb{R}) &
C^\bullet_b(\Sigma,\mathbb{R}) \ar[l]_-{i_1} &
C^\bullet_b(\Sigma,\partial \Sigma,\mathbb{R}) \ar[l]_-{j} \\
C^\bullet_b(S,S-K,\mathbb{R})\ar[u]^-{p_1} & C^\bullet_b(\Sigma, \Sigma-K, \mathbb{R}) \ar[l]_-{i_2} \ar[u]^-{p_2} \ar[ru]_-{p_3} &
}$$
where every map in the above diagram is the map induced from the canonical inclusion.
Every map in the diagram induces an isomorphism in bounded cohomology in degree $2$.
Thus, there exists a cocycle $z\in C^2_b(\Sigma,\Sigma-K,\mathbb{R})$ such that
$p_2(z)$ represents $\rho_b^*(\kappa_b)$ in $H^2_b(\Sigma,\mathbb{R})$ and
$i_1 (p_2(z))$ represents $\rho_b^*(\kappa_b)$ in $H^2_b(S,\mathbb{R})$ and
$p_3(z)$ represents $j^{-1}(\rho_b^*(\kappa_b))$ in $H^2_b(\Sigma,\partial \Sigma, \mathbb{R})$.
Here, we use the same notation $\rho^*_b(\kappa_b)$ for the bounded cohomology classes in $H^2_b(\Sigma,\mathbb{R})$ and $H^2_b(S,\mathbb{R})$ identified with $\rho^*_b(\kappa_b) \in H^2_b(\Gamma,\mathbb{R})$ via the canonical isomorphisms $H^2_b(\Sigma,\mathbb{R}) \cong H^2_b(\Gamma,\mathbb{R})$ and $H^2_b(S,\mathbb{R}) \cong H^2_b(\Gamma,\mathbb{R})$ respectively.

Let $c=\sum_{i=1}^\infty a_i \sigma_i$ be a locally finite fundamental $\ell^1$--cycle with finite Lipschitz constant representing $\alpha \in [S]^{\ell^1}_\mathrm{Lip}$. Then, we have
$$\langle \rho^*_b(\kappa_b),\alpha \rangle = \langle i_1 (p_2(z)), c \rangle = \langle z, c|_K \rangle,$$
where  $c|_K=\sum_{\mathrm{im}\sigma_i \cap K \neq \emptyset} a_i \sigma_i$. It is a standard fact that $c|_K$ represents the relative fundamental class $[S,S-K]$ in $H_2(S,S-K,\mathbb{R})$.
Since the fundamental cycle representing $[S,S-K]$ is also a representative of the fundamental class $[\Sigma,\Sigma-K]$ by the canonical inclusion, $c|_K$ represents the fundamental class $[\Sigma,\Sigma -K]$ in $H_2(\Sigma,\Sigma-K,\mathbb{R})$. Let $[z]$ denote the cohomology class in $H^2(\Sigma,\Sigma-K,\mathbb{R})$ determined by $z$. From the viewpoint of the Kronecker product $\langle \cdot, \cdot \rangle \co H^2(\Sigma, \Sigma-K, \br) \otimes H_2(\Sigma,\Sigma-K,\br) \rightarrow \br$, we have
$$\langle z, c|_K \rangle = \langle [z], [\Sigma, \Sigma-K] \rangle.$$

Let $d \in C_2(\Sigma,\partial \Sigma)$ be a cycle representing the fundamental cycle $[\Sigma,\partial \Sigma]$ in $H_2(\Sigma,\partial \Sigma,\mathbb{R})$. Since $p_3(z)$ represents $j^{-1}( \rho^*_b(\kappa_b))$,
$$ \langle j^{-1}(\rho_b^*(\kappa_b)), [\Sigma,\partial \Sigma] \rangle = \langle p_3(z), d \rangle = \langle z, d|_K \rangle.$$
For any relative fundamental cycle $d$ in $C_2(\Sigma,\partial \Sigma,\mathbb{R})$, $d|_K$ represents the fundamental class $[\Sigma, \Sigma-K]$ in $H_2(\Sigma,\Sigma-K,\mathbb{R})$.
Hence, $$\langle z,d|_K \rangle = \langle [z], [\Sigma,\Sigma-K] \rangle.$$
Therefore, we can finally conclude that
$$\langle \rho^*_b(\kappa_b),\alpha \rangle = \langle j^{-1}(\rho^*_b(\kappa_b)), [\Sigma,\partial \Sigma] \rangle = \langle [z], [\Sigma,\Sigma-K] \rangle,$$
which implies this proposition.
\end{proof}

The equation $\mathrm{Vol}(\rho)=2 \pi |\text{T}(\rho)|$ implies that the structure theorem for maximal representations of compact surfaces into $\text{PU}(1,1)$ with respect to the Toledo invariant $\text{T}(\rho)$ holds for the volume invariant $\mathrm{Vol}(\rho)$.

\begin{thm}\label{thm:6.3}
Let $\Gamma$ be a lattice in $\mathrm{PU}(1,1)$. Then, a representation $\rho \co \Gamma \rightarrow G$ is maximal if and only if $\rho$ is a discrete, faithful representation.
\end{thm}

Theorem \ref{thm:1.2} follows from Proposition \ref{pro:3.1},  Theorem \ref{thm:4.4}, \ref{thm:5.3} and \ref{thm:6.3}.

\begin{thm}\label{thm:6.4}
Let $\Gamma$ be an irreducible lattice in a connected semisimple Lie group $G$ with trivial center and
no compact factors. Let $\rho \co \Gamma \rightarrow G$ be a representation. Then, the volume invariant $\mathrm{Vol}(\rho)$ satisfies an inequality
$$ \mathrm{Vol}(\rho) \leq \mathrm{Vol}(M),$$
where $X$ is the associated symmetric space and $M=\Gamma\backslash X$.
Moreover,  equality holds if and only if
$\rho$ is a discrete, faithful representation.
\end{thm}

\section{Representations of lattices in $\mathrm{SO}(n,1)$ into $\mathrm{SO}(m,1)$}\label{sec:7}

In this section, we introduce a volume invariant $\mathrm{Vol}(\rho)$ for representations $\rho \co \Gamma \rightarrow \mathrm{SO}(m,1)$ of lattices $\Gamma$ in $\mathrm{SO}(n,1)$ for $m\geq n$.
Let $\mathbb{H}^k$ denote the hyperbolic $k$--space for each $k\in \mathbb{N}$. Define a map $f_n^m \co (\mathbb{H}^m)^{n+1} \rightarrow \mathbb{R}$ by
$$f_n^m(x_0,\ldots,x_n)= \mathrm{Vol}_n^m([x_0,\ldots,x_n]),$$
where $\mathrm{Vol}_n^m([x_0,\ldots,x_n])$ is the $n$--dimensional volume of the geodesic $n$--simplex $[x_0,\ldots,x_n]$ in $\mathbb{H}^m$. Clearly, $f_n^m$ is a $\mathrm{SO}(m,1)$--invariant continuous (bounded) cochain in $C^n_c(\mathbb{H}^m,\mathbb{R})$.
Observing that the geodesic $n$--simplex $[x_0,\ldots,x_n]$ is contained in a copy of $\mathbb{H}^n$ in $\mathbb{H}^m$, it is easy to see that $f_n^m$ is a continuous (bounded) cocycle and moreover,
$$\| \omega_n^m \|_\infty = v_n$$ where $\omega_n^m \in H^n_c(\mathrm{SO}(m,1),\mathbb{R})$ is the continuous cohomology class determined by the cocycle $f_n^m$ and $v_n$ is the volume of a regular ideal geodesic simplex in $\mathbb{H}^n$.

According to the Van Est isomorphism, the continuous cohomology class $\omega_n^m$ corresponds to a $\mathrm{SO}(m,1)$--invariant, differential $n$--form $\omega_n^m$ on $\mathbb{H}^m$.
The restriction of the differential form $\omega_n^m$ to any totally geodesic $\mathbb{H}^n$ in $\mathbb{H}^m$ is the Riemannian volume form on the totally geodesic $\mathbb{H}^n$ in $\mathbb{H}^m$.

Let $\Gamma$ be a lattice in $\mathrm{SO}(n,1)$ and $\rho \co \Gamma \rightarrow \mathrm{SO}(m,1)$ be a representation for $m\geq n$.
Let $c \co H^*_{c,b}(\mathrm{SO}(m,1),\mathbb{R}) \rightarrow  H^*_c(\mathrm{SO}(m,1),\mathbb{R})$ be the comparison map and $M=\Gamma \backslash \mathbb{H}^n$.
Then, we define a volume invariant $\mathrm{Vol}(\rho)$ of $\rho$ by
$$ \mathrm{Vol}(\rho) = \inf \{ | \langle \rho^*_b (\omega_{n,b}^m), \alpha \rangle | \ | \ c(\omega_{n,b}^m)=\omega_n^m \text{ and } \alpha \in [M]^{\ell^1}_\mathrm{Lip}\}.$$

It satisfies an inequality
$$ \mathrm{Vol}(\rho) \leq \| \omega_n^m \|_\infty \cdot \|M\|_\mathrm{Lip} = v_n \cdot \frac{\mathrm{Vol}(M)}{v_n}=\mathrm{Vol}(M).$$

Recall that a representation $\rho \co \Gamma \rightarrow \mathrm{SO}(m,1)$ is said to be a \emph{totally geodesic representation} if there is a totally geodesic $\mathbb{H}^n \subset \mathbb{H}^m$ so that the image of the representation lies in the subgroup $G \subset \mathrm{SO}(m,1)$ that preserves this $\mathbb{H}^n$
and that the $\rho$--equivariant map $F \co \mathbb{H}^n \rightarrow \mathbb{H}^m$ is a totally geodesic isometric embedding.
Note that the subgroup $G$ of $\mathrm{SO}(m,1)$ is of the form $H\times K$ where $H$ is isomorphic to $\mathrm{SO}(n,1)$ and $K$ is isomorphic to the compact group $\mathrm{SO}(m-n)$.
A totally geodesic representation $\rho \co \Gamma \rightarrow \mathrm{SO}(m,1)$ splits into $\rho = \rho_1 \times \rho_2$ where $\rho_1$ is conjugate to $\Gamma$ by the Mostow  rigidity theorem.

\begin{thm}
Let $\Gamma$ be a lattice in $\mathrm{SO}(n,1)$ and $M=\Gamma \backslash \mathbb{H}^n$. The volume invariant $\mathrm{Vol}(\rho)$ of a representation $\rho \co \Gamma \rightarrow \mathrm{SO}(m,1)$ for $m\geq n \geq 3$ satisfies an inequality $$\mathrm{Vol}(\rho) \leq \mathrm{Vol}(M).$$
Moreover, equality holds if and only if $\rho$ is a totally geodesic representation.
\end{thm}

\begin{proof}
We only need to show the second statement. In fact, a proof of the theorem is given by Bucher, Burger and Iozzi in \cite{BBI}. We give here an independent proof of the theorem for $m \geq n >3$. Suppose that a representation $\rho \co \Gamma \rightarrow \mathrm{SO}(m,1)$ is a totally geodesic representation.
Then, there exists a $\rho$--equivariant totally geodesic isometric embedding $F \co \mathbb{H}^n \rightarrow \mathbb{H}^m$.
The $\rho$--equivariant map $F$ induces homomorphisms $F^*_c \co H^\bullet_c(\mathrm{SO}(m,1),\mathbb{R})\rightarrow  H^\bullet (\Gamma,\mathbb{R})$ and
$F^*_b \co H^\bullet_{c,b}(\mathrm{SO}(m,1),\mathbb{R})\rightarrow  H^\bullet_b(\Gamma,\mathbb{R})$. The volume invariant $\mathrm{Vol}(\rho)$ of $\rho$ can be computed by
$$\mathrm{Vol}(\rho) = \inf \{ | \langle F^*_b (\omega_{n,b}^m), \alpha \rangle | \ | \ c(\omega_{n,b}^m)=\omega_n^m \text{ and } \alpha \in [M]^{\ell^1}_\mathrm{Lip}\}.$$

Since $F$ is an isometric embedding, we have
{\setlength\arraycolsep{2pt}
\begin{eqnarray*}
F^* f^m_n(y_0,\ldots,y_n) &=& f^m_n(F(y_0),\ldots,F(y_n)) \\
&=& \mathrm{Vol}_n^m([F(y_0),\ldots,F(y_n)]) \\
&=& \mathrm{sign}(F) \cdot \mathrm{Vol}^n_n([y_0,\ldots,y_n]),
\end{eqnarray*}}
where $\mathrm{sign}(F)=1$ if $F$ is orientation-preserving and $\mathrm{sign}(F)=-1$ if $F$ is orientation-reversing.
This implies that $F^*_c(\omega^m_n)=\mathrm{sign}(F) \cdot res_c(\omega_n)$ where $\omega_n$ is the $\mathrm{SO}(n,1)$--invariant volume form on $\mathbb{H}^n$. Hence, it immediately follows that $F^*_b(\omega^m_{n,b})= \mathrm{sign}(F) \cdot res_b(\omega_{n,b})$ for some $\omega_{n,b} \in c^{-1}(\omega_n)$.
By Lemma \ref{lem:4.2}, we have
$$ \langle F^*_b(\omega^m_{n,b}), \alpha \rangle = \langle \mathrm{sign}(F) \cdot res_b(\omega_{n,b}), \alpha \rangle = \mathrm{sign}(F) \cdot \mathrm{Vol}(M),$$
for all $\alpha \in [M]^{\ell^1}_\mathrm{Lip}$. Thus, we can conclude that $$ \mathrm{Vol}(\rho) = \inf \{ | \langle F^*_b (\omega_{n,b}^m), \alpha \rangle | \ | \ c(\omega_{n,b}^m)=\omega_n^m \text{ and } \alpha \in [M]^{\ell^1}_\mathrm{Lip}\} = \mathrm{Vol}(M).$$

Conversely, we suppose that $\mathrm{Vol}(\rho)=\mathrm{Vol}(M)$.
Recall that  the natural map $F \co \mathbb{H}^n \rightarrow \mathbb{H}^m$ satisfies:
\begin{itemize}
\item $F$ is smooth.
\item $F$ is $\rho$--equivariant.
\item For all $k\geq 3$, $\text{Jac}_kF(x) \leq (\delta(\Gamma)/(k-1))^k$.
\item If $\| D_xF(u_1)\wedge \cdots \wedge D_xF(u_k)\|= (\delta(\Gamma)/(k-1))^k$ for an orthonormal $k$-frame $u_1,\ldots,u_k$ at $x \in \mathbb{H}^n$, then the restriction of $D_xF$ to the subspace generated by $u_1,\ldots,u_k$ is a homothety.
\end{itemize}

Because of $\delta(\Gamma)=n-1$ for a lattice $\Gamma$ in $\mathrm{SO}(n,1)$, $\text{Jac}_nF(x)\leq 1$.
By an argument similar to the one used in the proof of Theorem \ref{thm:5.3}, we can conclude that
$$\left| \int_M F^*\omega_n^m \right| = \mathrm{Vol}(M).$$
Hence, $\mathrm{Jac}_n F(x) =1$ almost everywhere after possibly reversing the orientation of $X$.
Then, $F$ is a global isometry of $\mathbb{H}^n$. For a detailed proof about this, we refer to \cite{FK06}.
Therefore, $\rho$ is a totally geodesic representation.
\end{proof}

\section{Toledo invariant of complex hyperbolic representations}\label{sec:8}

In this section we consider only  uniform lattices $\Gamma\subset \mathrm{SU}(n,1),\ n\geq 2$.
\subsection{On complex hyperbolic space}
Let $\Gamma\subset \mathrm{SU}(n,1)$ be a uniform lattice that $M=\Gamma\backslash \bh^n_\bc$ and $\rho \co \Gamma\ra G=\mathrm{SU}(m,1),\ m\geq n$ be a  representation.
Let $\omega$ be a K\"ahler form on $\bh^m_\bc$. Then $\frac{1}{n!}\omega^n$ will be a $\mathrm{SU}(m,1)$-invariant form.
Then it defines an element $\omega_c\in H^{2n}_{c}(G,\br)$ via Van-Est isomorphism. Denote $\omega_b\in H^{2n}_{b,c}(G,\br)$ a bounded class such that $c(\omega_b)=\omega_c$ under the comparison map. Define the volume of the representation $\rho$ by
$$ \mathrm{Vol}(\rho) = \inf \{ |\langle \rho^*_b(\omega_b),\alpha \rangle| \ | \ c(\omega_b)=\omega_c \text{ and } \alpha \in [M]^{\ell^1}_\mathrm{Lip} \}.$$
Then, it satisfies the usual inequality
$$\mathrm{Vol}(\rho) \leq \|\omega_c\|_\infty \cdot \| M \|_\mathrm{Lip}.$$
But, since $\frac{1}{n!}\omega^n$ is the volume form on $\bh^n_\bc$, $\mathrm{Vol}(\rho) \leq \mathrm{Vol}(M)$.

Suppose $\mathrm{Vol}(\rho)=\mathrm{Vol}(M)$. If $\rho$ is not reductive, the image will be contained in a parabolic group, and the volume will be zero. Hence assume that $\rho$ is reductive.
Let $F \co \bh^n_\bc\ra \bh^m_\bc$ be a $\rho$--equivariant smooth harmonic map. Then
some class $\rho^*_b(\omega_b)$ is represented by $F^*(\frac{1}{n!}\omega^n)$ and the pairing satisfies
$$|\langle \rho^*_b(\omega_b),\alpha \rangle|=\left| \int_M F^* \left( \frac{1}{n!}\omega^n \right) \right| \geq\mathrm{Vol}(M).$$  This implies that the rank of $dF$ at some point $x\in \bh^n_\bc$ is maximal. By Siu's argument \cite{Siu}, $F$ is holomorphic. It is shown
in \cite{BCG99} that $\mathrm{Jac}_{2n}F\leq 1$ for holomorphic map $F$.
Consequently
$$\left| \int_M F^* \left( \frac{1}{n!}\omega^n \right) \right|
= \mathrm{Vol}(M)$$ and $F$ is an isometric embedding.
%This generalizes the inequality (3.13) of \cite{BCG99} without any assumption on $\rho$ (They assumed that $\rho$ is   faithful to get the contracting Jacobian).

%Suppose $\rho$ is faithful and $G:H^m_\bc\ra H^n_\bc$ is a $\rho^{-1}$-equivariant smooth map. Then
%$G\circ F:H^n_\bc\ra H^n_\bc$ is $\Gamma$--invariant and it induces a degree 1 map $H:M\ra M$. Hence at some point $y\in M$, $d_y F$ has a maximal rank. This implies that $F$ is holomorphic due to \cite{Siu}. If %$\gamma$ is a Ricci form of $H^m_\bc$,
%$\frac{1}{2\pi}F^*(\gamma)$ represents the first Chern class $c_1(F^*(T_\bc H^m_\bc))$. Since $\gamma=-2(m+1)\omega$,
%the integer
%$|\langle c_1^n(F^*(T_\bc H^m_\bc)), [M]^{\ell^1}_\mathrm{Lip} \rangle|$ is equal to
%$(\frac{m+1}{\pi})^n n! \mathrm{Vol}(\rho)$.  Hence $\mathrm{Vol}(\rho)$ takes discrete values.
Hence we obtain using the same proof of section \ref{sec:7} and the above argument
\begin{thm}Let $\Gamma\subset \mathrm{SU}(n,1)$ be a uniform lattice and $\rho \co \Gamma\ra \mathrm{SU}(m,1)$, $ m\geq n$ be a  representation. Then $\rho \co \Gamma\ra \mathrm{SU}(m,1)$ is a maximal volume representation if and only if $\rho$ is a totally geodesic representation.
\end{thm}
This is  a reformulation of Corlette's result in \cite{Corlette} in terms of the bounded cohomology theory.
See also \cite{KM} and \cite{BI} for defferent formulations.
Note that this theorem implies both Goldman-Millson and Corlette's results.
\begin{cor}\label{cor:8.2} Let $\Gamma\subset \mathrm{SU}(n,1)\subset \mathrm{SU}(m,1)$ be a uniform lattice. Then it is locally rigid up to compact group.
\end{cor}
\begin{proof}Suppose $\rho_t \co \Gamma\ra \mathrm{SU}(m,1)$ is an one-parameter family of representations near $\rho_0=i$, the canonical inclusion.
Note that $\mathrm{Vol}(\rho_t)=\left| \int_M f_t^* \left( \frac{1}{n!}\omega^n \right) \right| $ by Lemma \ref{lem:3.3} where $f_t$ is a $\rho_t$--equivariant map $f_t \co \bh^n_\bc\ra \bh^m_\bc$.

 Note that $f_0^*(\frac{1}{n!}\omega^n)\in H^{2n}(M,\bz)$ is a Chern class and hence $[f_t^*(\frac{1}{n!}\omega^n)]=[f_0^*(\frac{1}{n!}\omega^n)]\in H^{2n}(M,\bz)$. This implies that $\mathrm{Vol}(\rho_t)=\mathrm{Vol}(\rho_0)$. Since $\rho_0$ is a maximal volume representation, $\rho_t$ is also maximal, hence
 they are all conjugate each other up to compact group.
\end{proof}

\subsection{On quaternionic hyperbolic space}
We can also formulate the rigidity phenomenon of  uniform lattices of $\mathrm{SU}(n,1)$ in $\mathrm{Sp}(m,1)$  as follows.
A homogeneous space $D=\mathrm{Sp}(m,1)/\mathrm{Sp}(m)\times\mathrm{U}(1)$ over $\bh^m_\bh$ with fiber $\bc P^1$ is called a twistor space. The vertical bundle $\cal V$ tangent to the fibers is a smooth subbundle of $TD$ and a unique $\mathrm{Sp}(m,1)$-invariant complement to this vertical subbundle is called the holomorphic
horizontal subbundle $\cal H$. This twistor space $D$ posseses a pseudo-K\"ahler metric $g$ which is negative definite on $\cal V$ and positive definite on $\cal H$, whose associated (1,1) form $\omega_D$ is closed.
The quaternionic hyperbolic space $\bh^m_\bh$ has one dimensional space of $\mathrm{Sp}(m,1)$--invariant four forms. We pick a four form $\alpha$ so that its restriction to a totally geodesic complex hyperbolic subspace $\bh^m_\bc$ is $\omega^2_{\bh^m_\bc}$ where $\omega_{\bh^m_\bc}$ is a K\"ahler form on complex hyperbolic space. The relation between $\alpha$ and $\omega_D$ is as follows, see \cite[Lemma 1]{DG}.
$$\pi^*\alpha=\omega_D^2 + d\beta$$ for some $\mathrm{Sp}(m,1)$--invariant 3--form $\beta$. Note that $\omega_D^n\in H^{2n}_{c}(\mathrm{Sp}(m,1),\br)$.
One can define the volume of a representation $\rho \co \Gamma\ra \mathrm{Sp}(m,1)$ by$$ \mathrm{Vol}(\rho) = \inf \{ |\langle \rho^*_b(\omega_b),\beta \rangle| \ | \ c(\omega_b)=\omega^n_D \text{ and } \beta \in [M]^{\ell^1}_\mathrm{Lip} \}.$$
Then $\mathrm{Vol}(\rho)\leq \| \omega^n_D \|_\infty \cdot \| M \|_{\mathrm{Lip}}$.
Suppose $\mathrm{Vol}(\rho)=\mathrm{Vol}(M)$. Note that such a value is realized when $\rho$ is a totally geodesic embedding since the restriction of $\omega_D$ to $\bh^n_\bc$ is a K\"ahler form. Let $f\co \bh^n_\bc\ra \bh^m_\bh$ be a $\rho$--equivariant harmonic map and let
$F$ be a lift of $f$ to the twister space $D$ so that $\pi\circ F=f$.
If $f^*\alpha=0$, then $F^*\omega_D^2=-dF^*\beta$ on $M$ and thus, we have $$\int_M F^*\omega^n_D =-\int_M F^*(\omega^{n-2}_D\wedge d\beta)=-\int_M dF^*(\omega^{n-2}_D\wedge \beta)=0.$$

Hence we may assume that $f^*\alpha\neq 0$, then the rank of $f$ is at least four at some point. By \cite{CD}, one can choose  $F$ to be a holomorphic horizontal lift.
 Then $F^* \omega_D^n$ represents some class
$\rho^*_b(\omega_b)$ with $c(\omega_b)=\omega^n_D$,  and
$$\mathrm{Vol}(\rho)=\mathrm{Vol}(M)\leq \int_M F^*\omega_D^n. $$ Since $F$ is holomorphic, $F^*\omega_D\leq \omega_M$, hence
$$\int_M F^*\omega_D^n\leq \int_M \omega_M^n=\mathrm{Vol}(M).$$
This forces $F$ to be totally geodesic embedding.
Hence the local rigidity of
$\Gamma\subset \mathrm{SU}(n,1)\subset \mathrm{Sp}(m,1)$ also follows as in Corollary \ref{cor:8.2}, which is part of a result in \cite{KKP}.

\end{document}